\documentclass[11pt,letterpaper]{amsart}
\usepackage[utf8]{inputenc}\usepackage[left=2.7cm,right=2.7cm,top=3.5cm,bottom=3cm]{geometry}
\usepackage{amssymb,latexsym,amsmath,amsthm,amscd}
\usepackage{graphicx}
\usepackage{dsfont}
\usepackage{stmaryrd}
\usepackage[pagebackref=false,colorlinks]{hyperref}
\hypersetup{pdffitwindow=true,linkcolor=blue,citecolor=blue,urlcolor=cyan}
\usepackage{hyperref}
\usepackage[all]{xy}
\usepackage{mathrsfs}
\usepackage{color}
\definecolor{Blue}{rgb}{0.3,0.3,0.9}

\usepackage[T2A,OT1]{fontenc}
\DeclareSymbolFont{cyrillic}{T2A}{cmr}{m}{n}
\DeclareMathSymbol{\Sha}{\mathalpha}{cyrillic}{216}

\theoremstyle{plain}
\newtheorem{thm}{Theorem}[subsection] 

\usepackage{mathrsfs}

\theoremstyle{definition}

\newtheorem{defn}[thm]{Definition}

\newtheorem{rem}[thm]{Remark}
\newtheorem{choice}[thm]{Choice}
\theoremstyle{definition}

\theoremstyle{plain}
\newtheorem{prop}[thm]{Proposition}
\theoremstyle{plain}
\newtheorem{lem}[thm]{Lemma}
\theoremstyle{plain}
\newtheorem{cor}[thm]{Corollary}
\newtheorem*{thm-intro}{Theorem}
\newtheorem*{cor-intro}{Corollary}

\numberwithin{equation}{section}

\newcommand{\eps}{\varepsilon}

\newcommand{\cc}{\mathbf{c}}

\newcommand{\rH}{{\mathrm{H}}}

\newcommand{\pp}{\mathfrak{p}}
\newcommand{\ppbar}{\overline{\mathfrak{p}}}
\newcommand{\qq}{\mathfrak{q}}
\newcommand{\fa}{\mathfrak{a}}
\newcommand{\bQ}{\mathbf{Q}}
\newcommand{\bZ}{\mathbf{Z}}
\newcommand{\bC}{\mathbf{C}}

\def\cO{\mathcal O}

\newcommand{\VQdag}{{\mathbf{V}_{\underline{Q}}^\dagger}}
\newcommand{\Vsdag}{{\mathbb{V}_{f\boldsymbol{gg}^*}^\dagger}}
\newcommand{\Asdag}{{\mathbb{A}_{f\boldsymbol{gg}^*}^\dagger}}

\newcommand{\Vsdagg}{{\mathbb{V}_{f\underline{\boldsymbol{gg}}^*}^\dagger}}
\newcommand{\Asdagg}{{\mathbb{A}_{f\underline{\boldsymbol{gg}}^*}^\dagger}}
\newcommand{\ggstar}{\underline{\boldsymbol{gg}}^*}

\newcommand{\Vdag}{{\mathbf{V}^\dagger}}
\newcommand{\Adag}{{\mathbf{A}^\dagger}}
\newcommand{\VVdag}{{\mathbb{V}^\dagger}}
\newcommand{\AAdag}{{\mathbb{A}^\dagger}}

\newcommand{\unb}{{\boldsymbol{f}}}

\newcommand{\Q}{\mathbf{Q}}
\newcommand{\Z}{\mathbf{Z}}

\def\makeop#1{\expandafter\def\csname#1\endcsname
	{\mathop{\rm #1}\nolimits}\ignorespaces}
\makeop{Hom}   \makeop{End}   \makeop{Aut}   
\makeop{Pic} \makeop{Gal}       \makeop{Div} \makeop{Lie}
\makeop{PGL}   \makeop{Corr} \makeop{PSL} \makeop{sgn} \makeop{Spf}
\makeop{Tr} \makeop{Nr} \makeop{Fr} \makeop{disc}
\makeop{Proj} \makeop{supp} \makeop{ker}   \makeop{Im} \makeop{dom}
\makeop{coker} \makeop{Stab} \makeop{SO} \makeop{SL} \makeop{SL}
\makeop{Cl}    \makeop{cond} \makeop{Br} \makeop{inv} \makeop{rank}
\makeop{id}    \makeop{Fil} \makeop{Frac}  \makeop{GL} \makeop{SU}
\makeop{Trd}   \makeop{Sp} \makeop{Tr}    \makeop{Trd} \makeop{Res}
\makeop{ind} \makeop{depth} \makeop{Tr} \makeop{st} \makeop{Ad}
\makeop{Int} \makeop{tr}    \makeop{Sym} \makeop{can} \makeop{SO}
\makeop{torsion} \makeop{GSp} \makeop{Tor}\makeop{Ker} \makeop{rec}
\makeop{Ind} \makeop{Coker}
\makeop{vol} \makeop{Ext} \makeop{gr} \makeop{ad}
\makeop{Gr}\makeop{corank} \makeop{Ann}
\makeop{Hol} 
\makeop{Fitt} \makeop{Mp} \makeop{CAP}

\def\x{{\times}}

\def\al{\alpha}

\def\iso{\simeq}

\newcommand{\dBr}[1]{\llbracket{#1}\rrbracket}

\newcommand{\bT}{\mathbb{T}}

\newcommand{\cR}{\mathbb{I}}

\newcommand{\any}{?}

\newcommand{\bfff}{{\boldsymbol{f}}}
\newcommand{\bff}{{\boldsymbol{f}}}
\newcommand{\bfg}{{\boldsymbol{g}}}
\newcommand{\bfh}{{\boldsymbol{h}}}

\begin{document}
	
\title[Nonvanishing of generalised Kato classes and Iwasawa main conjectures]{Nonvanishing of generalised Kato classes and \\Iwasawa main conjectures}
\author[F.\,Castella]{Francesc Castella}
	
\date{\today}
		
\address[]{Department of Mathematics, University of California Santa Barbara, CA 93106, USA}
\email{castella@ucsb.edu}




\begin{abstract}
A construction due to Darmon--Rotger gives rise to generalised Kato classes $\kappa_p(E)$ in the $p$-adic Selmer group ${\rm Sel}(\Q,V_pE)$ of elliptic curves $E/\Q$ of positive even analytic rank, where $p>3$ is any  prime of good ordinary reduction for $E$. In \cite{DR2.5}, they  conjectured that
$\kappa_p(E)\neq 0$ precisely when ${\rm Sel}(\Q,V_pE)$ is two-dimensional. The first cases of this conjecture were obtained by the author with M.-L.\,Hsieh \cite{cas-hsieh-ord}. 
In this note  we give a new proof of the implication
\[
\kappa_p(E)\neq 0\quad\Longrightarrow\quad{\rm dim}_{\Q_p}{\rm Sel}(\Q,V_pE)=2
\]
established in \emph{op.\,cit.}, and show that the converse implication holds if \emph{and only if} the restriction map ${\rm loc}_p:{\rm Sel}(\Q,V_pE)\rightarrow E(\Q_p)\hat\otimes\Q_p$ is nonzero. The present approach is an adaptation to the non-CM case of the method introduced by the author  \cite{GKC-cm} in the case of CM elliptic curves.
\end{abstract}


\maketitle

\section{Introduction}

After the grounbreaking works of Gross--Zagier and Kolyvagin in the 1980s, the construction of non-torsion rational points (or more generally, Selmer classes) on elliptic curves of $E/\Q$ with ${\rm ord}_{s=1}L(E,s)\geq 2$, akin to the construction of Heegner points in the cases of analytic rank $0$ or $1$, is widely regarded as one of the central open problems in number theory.

In a recent series of spectacular works \cite{DR1,DR2,DR3,BSV0,BSV} (culminating in the collective volume \cite{BDRSV} with applications to the theory of Stark--Heegner points), Darmon--Rotger and Bertolini--Seveso--Venerucci revisited the construction of diagonal cycle classes due to Gross--Kudla \cite{gross-kudla} and Gross--Schoen \cite{gross-schoen}, obtaining in particular an interpolation of these classes in $p$-adic families. Directly from their geometric construction, one obtains classes attached to triples $(f,g,h)$ of cuspidal eigenforms of ``balanced'' weights $(k,l,m)$ (meaning that none of $k$, $l$, or $m$ is larger than the sum of the other two),  
but after $p$-adic interpolation one also gets classes for weights beyond this range, such as the ``unbalanced'' weight $(2,1,1)$. 

Of special relevance to the Birch--Swinnerton-Dyer conjecture and its equivariant refinements is the case where $f$ is the newform of weight $2$ associated  to an elliptic curve $E/\Q$, and $g$ and  $h$ are the weight $1$ cuspidal eigenforms associated (as a consequence of the proof of Serre's conjecture \cite{KW-I}) to degree $2$ odd irreducible Artin representations $\varrho_1$ and $\varrho_2$, respectively. In this case, one of the main results of \cite{DR2} and \cite{BSV} relates the resulting \emph{generalised Kato classes} 
\[
\kappa_p(f,g,h)\in\rH^1(\Q,V_pE\otimes\varrho),\quad\varrho:=\varrho_1\otimes\varrho_2,
\] 
obtained from a specialisation in weights $(2,1,1)$ of a $p$-adic family of diagonal cycles classes, 
to the value at $s=1$ of the twisted Hasse--Weil $L$-funcion $L(E\otimes\varrho,s)$. Together with global duality, this relation 
leads to a proof of the implication
\[
L(E\otimes\varrho,1)\neq 0\quad\Longrightarrow\quad{\rm Hom}_{G_\Q}(V_{\varrho}^\vee,E(H)\otimes L)=0,
\]
where $L$ is any number field over which $\varrho$ is defined (assumed to admit an embedding into $\Q_p$ for simplicity), $V_{\varrho}^\vee$ is the linear dual of an $L$-vector space affording $\varrho$, and $H$ is the fixed field of $\ker(\varrho)$. More precisely, by virtue of the explicit reciprocity law obtained in \emph{loc.\,cit.}, the nonvanishing of $L(E\otimes\varrho,1)$ implies that the class $\kappa_p(f,g,h)$ is \emph{non-crystalline} at $p$, 
from where the bound on the $\varrho$-isotypical component of $E(H)$ follows by global duality. 

Interestingly, the same explicit reciprocity law shows that the generalised Kato classes $\kappa_p(f,g,h)$ are Selmer whenever $L(E\otimes\varrho,1)=0$. Moreover, the representations $\varrho$ for which the construction applies necessarily have real traces, and in many cases the sign in the functional of the self-dual $L$-function $L(E\otimes\varrho,s)$ is $+1$. Thus the results of \cite{DR2} and \cite{BSV} provide a construction of (possibly zero \emph{a priori}!) Selmer classes in situations where ${\rm ord}_{s=1}L(E\otimes\varrho,s)\geq 2$. 

In \cite{DR2.5}, Darmon--Rotger carried out a systematic study of their construction in relation with the Birch--Swinnerton-Dyer conjecture and the elliptic Stark conjecture of \cite{DLR-stark}. (More recently, a similar study was carried out by Bertolini--Seveso--Venerucci \cite{BSV-garrett}, which in particular allows one to interpret and refine some of the rationality conjectures in \cite{DR2.5} in terms of a $p$-adic Birch--Swinnerton-Dyer conjecture for $p$-adic Garrett--Rankin $L$-functions.) In particular, when $L(E\otimes\varrho,1)=0$ with sign $+1$ they conjectured the equivalence
\begin{equation}\label{eq:nonv-crit}
\kappa_p(f,g,h)\neq 0\quad\overset{?}\Longleftrightarrow\quad{\rm dim}_{\Q_p}{\rm Sel}(\Q,V_pE\otimes\varrho)=2.
\end{equation} 
This leads to the expectation, when combined with the equivariant Birch--Swinnerton-Dyer conjecture, that $\kappa_p(f,g,h)$ is a nonzero class in ${\rm Sel}(\Q,V_pE\otimes V_\varrho)$ if and only if ${\rm ord}_{s=1}L(E\otimes\varrho,s)=2$. 

Of special interest 
to the Birch--Swinnerton-Dyer conjecture is the case in which $\varrho$ contains the trivial representation. 
This occurs precisely when $\varrho_2\simeq\varrho_1^\vee$, so that
\begin{equation}\label{eq:dec-Vintro}
V_{\varrho}\simeq L\oplus{\rm ad}^0(V_{\varrho_1}),\nonumber
\end{equation}
where ${\rm ad}^0(V_{\varrho_1})$ is the $3$-dimensional representation of $G_\Q$ acting on the trace zero endomorphisms of $V_{\varrho_1}$. By the Artin formalism, we then have
\[
L(E\otimes\varrho,s)=L(E,s)\cdot L(E\otimes{\rm ad}^0(\varrho_1)).
\]
In particular, if $L(E\otimes{\rm ad}^0(\varrho_1))\neq 0$, denoting by $\kappa_p(E)$ the image of $\kappa_p(f,g,h)$ under the resulting  natural projection $\rH^1(\Q,V_pE\otimes\varrho)\rightarrow\rH^1(\Q,V_pE)$,  
by the Bloch--Kato conjecture \cite{BK} 
(which in this case predicts that the dimensions of the Selmer groups for $V_pE$ and $V_pE\otimes\varrho$ are the same) 
the equivalence  (\ref{eq:nonv-crit}) amounts to the conjectural equivalence
\begin{equation}\label{eq:nonv-crit-2}
\kappa_p(E)\neq 0\quad\overset{?}\Longleftrightarrow\quad{\rm dim}_{\Q_p}{\rm Sel}(\Q,V_pE)=2.
\end{equation}
%

In \cite{cas-hsieh-ord}, the author and M.-L.\,Hsieh proved the 
first cases of conjecture \eqref{eq:nonv-crit-2}, in situations where $\varrho_1$ is induced from a finite order character of an imaginary quadratic field in which $p$ splits. The key ingredient in the proof 
was a leading coefficient formula for the $p$-adic $L$-function introduced in \cite{BDmumford-tate} in terms of Howard's derived $p$-adic heights. 
\cite{howard-derived}. 

The main result of this note (see Theorem\,\ref{thm:main}) includes a new proof of \cite[Thm.\,A]{cas-hsieh-ord}. The approach is an adaptation of the method introduced by the author in \cite{GKC-cm}, where the case of CM elliptic curves $E/\Q$ is treated, and clarifies the role played by the restriction map at $p$
\[
{\rm loc}_p:{\rm Sel}(\Q,V_pE)\rightarrow E(\bQ_p)\hat\otimes\Q_p.
\]
Indeed, in \emph{loc.\,cit.} it was shown that (under mild hypotheses) the nonvanishing of $\kappa_p(E)$ implies that ${\rm Sel}(\bQ,V_pE)$ is $2$-dimensional, while for the proof of the converse implication 
it was necessary to assume that 
\begin{equation}\label{eq:locp}
{\rm Sel}(\Q,V_pE)\neq{\rm ker}({\rm loc}_p)\tag{${\rm Loc}_p$}. 
\end{equation}
%
A novel insight of the approach in this paper is that, assuming ${\rm dim}_{\Q_p}{\rm Sel}(\bQ,V_pE)=2$, condition \eqref{eq:locp} is in fact \emph{necessary} for the nonvanishing of $\kappa_p(E)$.  


Similarly as in \cite{GKC-cm}, our approach is based on a study of the relation between two different formulations of the Iwasawa--Greenberg main conjecture for triple products studied in 
\cite{ACR}: 
\begin{enumerate}
\item[(IMC-1)] one in terms of Hsieh's triple product $p$-adic $L$-functions \cite{hsieh-triple}; 
\item[(IMC-2)] another  without reference to $p$-adic $L$-functions, phrased in terms of the $p$-adic family of diagonal cycles used in the construction of $\kappa_p(f,g,h)$ and $\kappa_p(E)$.
\end{enumerate}

In the case where $\varrho_1=\varrho_2^\vee$ is dual to $\varrho_2$, and  is induced from a finite order Hecke character of an imaginary quadratic field $K$, conjecture (IMC-1) can be related to the Iwasawa main conjecture for $E/K$ in the anticyclotomic setting, and using global duality and the explicit reciprocity laws of \cite{DR3} and \cite{BSV}, one can easily show that (in general) conjectures (IMC-1) and (IMC-2) are equivalent. Thus, from the results on the anticyclotomic main conjecture stemming from the works of Bertolini--Darmon \cite{bdIMC} and Skinner--Urban \cite{SU}, we deduce under mild hypotheses a proof of conjecture (IMC-2), from where the 
proof of our main result follow easily from a variant of Mazur's control theorem and global duality.

We conclude this introduction by noting that even though throughout the paper $p$ is assumed to be a  prime of \emph{good} ordinary reduction, it seems possible to extend our results to the multiplicative case; it would then be interesting to compare the resulting $\kappa_p(E)$ with the 
$p$-adic limits of Heegner points studied in \cite{DF,FG-mock} (see also \cite{cas-3split} for a related construction).

\subsection{Acknowledgements} 


I heartily thank Matteo Longo, Marco Adamo Seveso, Rodolfo Venerucci, and Stefano Vigni for the opportunity to contribute this note 
to the proceedings in celebration of Massimo Bertolini's 60th birthday. The chief insight exploited in this paper 
is that the nonvanishing of $\kappa_p(E)$ can be related to an Iwasawa main conjecture in the spirit of the Heegner point main conjecture studied by Massimo about 30 years ago \cite{bertolini-PhD}. Some of Massimo's results since then also play a key role in this paper, and it is a pleasure to dedicate it to him as an attestation of his lasting and growing impact on the subject. During the preparation of this note, the author was partially supported by the NSF grants DMS-1946136 and DMS-2101458.

\section{$p$-adic $L$-functions}

Fix a prime $p>3$. Let $K$ be an imaginary quadratic field of discriminant $D_K$ and assume that 
\[
\textrm{$(p)=\pp\ppbar$ splits in $K$}, 
\]
with $\pp$ the prime of $K$ above $p$ induced by a fixed embedding $\iota_p:\overline{\Q}\hookrightarrow\overline{\Q}_p$. In this section we recall from \cite{hsieh-triple} the construction of the triple product $p$-adic $L$-function for Hida families, and its relation with the anticyclotomic $p$-adic $L$-functions of Bertolini--Darmon \cite{BDmumford-tate} in the case where two of the Hida families are CM Hida families attached to dual ray class characters of $K$.

\subsection{Triple product $p$-adic $L$-function}\label{subsec:triple}


Let $\cR$ be a normal domain finite flat over 
\[
\Lambda:=\mathscr{O}\dBr{1+p\Z_p},
\] 
where $\mathscr{O}$ is the ring of integers of a finite extension of $\Q_p$. For a positive integer $N$ prime $p$ and a Dirichlet character $\chi:(\Z/Np\Z)^\times\rightarrow\mathscr{O}^\times$, we denote by $S^o(N,\chi,\cR)\subset\cR\dBr{q}$ the space of ordinary $\cR$-adic cusp forms of tame level $N$ and branch character $\chi$ as defined in \cite[\S{3.1}]{hsieh-triple}. 

Denote by $\mathfrak{X}_\cR^+\subset{\rm Spec}\,\cR(\overline{\Q}_p)$ the set of \emph{arithmetic points} of $\cR$, consisting of the ring homomorphisms $Q:\cR\rightarrow\overline{\Q}_p$ such that $Q\vert_{1+p\Z_p}$ is given by $z\mapsto z^{k_Q}\epsilon_Q(z)$ for some $k_Q\in\Z_{\geq 2}$ called the \emph{weight of $Q$} and $\epsilon_Q(z)\in\mu_{p^\infty}$. As in  \cite[\S{3.1}]{hsieh-triple}, we say that $\boldsymbol{f}=\sum_{n=1}^\infty a_n(\bfff)q^n\in S^o(N,\chi,\cR)$ is a \emph{primitive Hida family} if for every $Q\in\mathfrak{X}_\cR^+$ the specialization $\boldsymbol{f}_Q$ gives the $q$-expansion of an ordinary $p$-stabilised newform of weight  $k_Q$ and tame conductor $N$. Attached to such $\bfff$ we let $\mathfrak{X}_{\cR}^{\rm cls}\subset\mathfrak{X}_{\cR}^+$ be the set of ring homomorphisms $Q$ as above with $k_Q\in\Z$ such that $\bfff_Q$ is the $q$-expansion of a classical modular form.

For $\bfff$ a primitive Hida family of tame level $N$, we let 
\[
\rho_{\bff}:G_{\Q}\rightarrow{\rm GL}_2({\rm Frac}\,\cR)
\] 
denote the associated Galois representation, where ${\rm Frac}\,\cR$ is the field of fractions of $\cR$. It will be convenient for us to take $\rho_{\bff}$ to be the \emph{dual} of that in \cite[\S{3.2}]{hsieh-triple}; in particular, $\det\,\rho_\bfff=\chi_\cR\cdot\varepsilon_{\rm cyc}$ in the notations of \emph{loc.\,cit.}, where $\varepsilon_{\rm cyc}$ is the $p$-adic cyclotomic character. We assume that the associated the residual representation $\bar{\rho}_\bfff$ is absolutely irreducible, and also denote by 
\[
\rho_{\bfff}:G_\Q\rightarrow{\rm Aut}_{\cR}(V_\bfff)\simeq{\rm GL}_2(\cR)
\]
the realisation of $\rho_\bff$ on a module $V_\bff\simeq\cR^{\oplus 2}$.  
  By \cite[Thm.~2.2.2]{wiles88}, restricted to $G_{\Q_p}$ the Galois representation $V_\bfff$ fits into a short exact sequence
\[
0\rightarrow V_\bfff^+\rightarrow V_\bfff\rightarrow V_\bfff^-\rightarrow 0,
\]
where $V_\bfff^-$ is free of rank $1$ over $\cR$, with the $G_{\Q_p}$-action given by the unramified character sending an arithmetic Frobenius ${\rm Fr}_p$ to $a_p(\bfff)$. 

Associated with $\bfff$ there is a $\cR$-algebra homomorphism 
\[
\lambda_{\bfff}:\bT(N,\cR)\rightarrow\cR
\] 
where $\bT(N,\cR)$ is the Hecke algebra acting on $\oplus_\chi S^o(N,\chi,\cR)$, where $\chi$ runs over the characters of $(\Z/pN\Z)^\times$. Let $\bT_{\mathfrak{m}}$ be the local component of $\bT(N,\cR)$ through which $\lambda_{\bfff}$ factors, and following \cite{hida-AJM-modules} define the \emph{congruence ideal} $C(\bfff)$ of $\bfff$ by
\[
C(\bfff):=\lambda_{\bfff}({\rm Ann}_{\bT_\mathfrak{m}}({\rm ker}\,\lambda_\bfff))\subset\cR.
\] 
When, in addition to absolutely irreducible, $\bar{\rho}_{\bfff}$ is also $p$-distinguished (i.e., the semi-simplification of $\bar{\rho}_{\bfff}\vert_{G_{\Q_p}}$ is non-scalar), it follows from the results of \cite{Fermat-Wiles} and \cite{hida-AJM-modules} that $C(\bfff)$ is generated by a nonzero element $\eta_{\bfff}\in\cR$.

\subsubsection{Triple products of Hida families}\label{subsubsec:triple-hida}

Let
\[
(\bff,\bfg,\bfh)\in S^o(N_f,\chi_f,\cR_f)\times S^o(N_g,\chi_g,\cR_g)\times S^o(N_h,\chi_h,\cR_h)
\]
be a triple of primitive Hida families with 
\begin{equation}\label{eq:a}
\textrm{$\chi_f\chi_g\chi_h=\omega^{2a}$ for some $a\in\Z$,}
\end{equation} 
where $\omega$ is the Teichm\"uller character. Put 
\[
\mathcal{R}=\cR_f\hat\otimes_{\mathscr{O}}\cR_g\hat\otimes_{\mathscr{O}}\cR_h,
\] 
which is a finite extension of the three-variable Iwasawa algebra $\Lambda\hat\otimes_{\mathscr{O}}\Lambda\hat\otimes_{\mathscr{O}}\Lambda$, 
and let
\[
\mathfrak{X}_{\mathcal{R}}^\bff:=\left\{(Q_1,Q_2,Q_3)\in\mathfrak{X}_{\cR_\varphi}^+\times\mathfrak{X}_{\cR_g}^{\rm cls}\times\mathfrak{X}_{\cR_h}^{\rm cls}\;:\;\textrm{$k_{Q_1}\geq k_{Q_2}+k_{Q_3}$ and $k_{Q_1}\equiv k_{Q_2}+k_{Q_3}\;({\rm mod}\;2)$}\right\}
\]
be the weight space for $\mathcal{R}$ in the so-called \emph{$\bff$-unbalanced range}. 

Let $\mathbf{V}=V_\bff\hat\otimes_{\mathscr{O}}V_{\bfg}\hat\otimes_{\mathscr{O}}V_{\bfh}$ be the triple tensor product Galois representation attached to $(\bff,\bfg,\bfh)$, and writing $\det\mathbf{V}=\mathcal{X}^2\varepsilon_{\rm cyc}$ (as is possible by \eqref{eq:a}) define  
\begin{equation}\label{eq:crit-twist}
\Vdag:=\mathbf{V}\otimes\mathcal{X}^{-1},
\end{equation}
which is a self-dual twist of $\mathbf{V}$. Define the rank four $G_{\Q_p}$-invariant subspace $\mathscr{F}_p^\bff(\Vdag)\subset\Vdag$ by
\begin{equation}\label{eq:unb-intro}
\mathscr{F}_p^\bff(\Vdag):=V_{\bff}^+\hat\otimes_{\cO}V_\bfg\hat\otimes_{\cO}V_{\bfh}\otimes\mathcal{X}^{-1}.
\end{equation}
For every $\underline{Q}=(Q_1,Q_2,Q_3)\in\mathfrak{X}_{\mathcal{R}}^\bff$ we denote by $\mathscr{F}_p^\bff(\VQdag)\subset\VQdag$  the corresponding specialisations. Finally, for every rational prime $\ell$ denote by $\varepsilon_\ell(\VQdag)$ the epsilon factor attached to the restriction of $\VQdag$ to $G_{\Q_\ell}$ as in \cite[p.\,21]{tate-background}, and assume that 
\begin{equation}\label{eq:+1}
\textrm{for some $\underline{Q}\in\mathfrak{X}_{\mathcal{R}}^\bff$, we have $\varepsilon_\ell(\VQdag)=+1$ for all prime factors $\ell$ of $N_f N_g N_h$.}
\end{equation}
As explained in \cite[\S{1.2}]{hsieh-triple}, condition (\ref{eq:+1}) is independent of $\underline{Q}$, and it implies that the sign in the functional equation for the triple product $L$-function 
\[
L(\VQdag,s) 
\]
(with center at $s=0$) is $+1$ for all $\underline{Q}\in\mathfrak{X}_{\mathcal{R}}^\bff$.

\begin{thm}\label{thm:hsieh-triple}
Let $(\bff,\bfg,\bfh)$ be a triple of primitive Hida families as above satisfying conditions (\ref{eq:a}) and (\ref{eq:+1}). Assume in addition that:
\begin{itemize}
\item $\gcd(N_f,N_g,N_h)$ is square-free,
\item the residual representation $\bar{\rho}_{\bff}$ is absolutely irreducible and $p$-distinguished,
\end{itemize}
and fix a generator $\eta_{\bff}$ of the congruence ideal of $\bff$. Then there exists a unique element 
\[
\mathscr{L}_p^{\bff}(\bff,\bfg,\bfh)\in\mathcal{R}
\]
such that for all $\underline{Q}=(Q_0,Q_1,Q_2)\in\mathfrak{X}_{\mathcal{R}}^{\bff}$ of weight $(k_0,k_1,k_2)$ with $\epsilon_{Q_0}=1$ we have
\[
(\mathscr{L}_p^\bff(\bff,\bfg,\bfh)(\underline{Q}))^2=\Gamma_{\VQdag}(0)\cdot\frac{L(\VQdag,0)}{(\sqrt{-1})^{2k_{0}}\cdot\Omega_{\bff_{Q_0}}^2}\cdot\mathcal{E}_p(\mathscr{F}_p^{\bff}(\VQdag))\cdot\prod_{\ell\in\Sigma_{\rm exc}}(1+\ell^{-1})^2,
\]
where:
\begin{itemize}
\item $\Gamma_{\VQdag}(0)=\Gamma_{\bC}(c_{\underline{Q}})\Gamma_\bC(c_{\underline{Q}}+2-k_1-k_2)\Gamma_\bC(c_{\underline{Q}}+1-k_1)\Gamma_\bC(c_{\underline{Q}}+1-k_2)$, with 
\[
c_{\underline{Q}}=(k_0+k_1+k_2-2)/2
\] 
and $\Gamma_\bC(s)=2(2\pi)^{-s}\Gamma(s)$;
\item $\Omega_{\bff_{Q_0}}$ is the canonical period
\[
\Omega_{\bff_{Q_0}}:=(-2\sqrt{-1})^{k_0+1}\cdot\frac{\Vert\bff_{Q_0}^\circ\Vert_{\Gamma_0(N_f)}^2}{\iota_p(\eta_{\bff_{Q_0}})}\cdot\Bigl(1-\frac{\chi_{f}'(p)p^{k_0-1}}{\alpha_{Q_0}^2}\Bigr)\Bigl(1-\frac{\chi_{f}'(p)p^{k_0-2}}{\alpha_{Q_0}^2}\Bigr),
\]
with $\bff_{Q_0}^\circ\in S_{k_0}(N_{f})$ the newform of conductor $N_f$ associated with $\bff_{Q_0}$, $\chi_f'$ the prime-to-$p$ part of $\chi_f$, and $\alpha_{Q_0}$ the specialization of $a_p(\bff)\in\cR_f^\times$ at $Q_0$;
\item $\mathcal{E}_p(\mathscr{F}_p^{\bff}(\VQdag))$ is the modified $p$-Euler factor
\[
\mathcal{E}_p(\mathscr{F}_p^{\bff}(\VQdag)):=\frac{L_p(\mathscr{F}_p^{\bff}(\VQdag),0)}{\varepsilon_p(\mathscr{F}_p^{\bff}(\VQdag))\cdot L_p(\VQdag/\mathscr{F}_p^{\bff}(\VQdag),0)}\cdot\frac{1}{L_p(\VQdag,0)},
\]
\end{itemize}
and $\Sigma_{\rm exc}$ is an explicitly defined subset of the prime factors of $N_f N_g N_h$, \cite[p.\,416]{hsieh-triple}.
\end{thm}

\begin{proof}
This is \cite[Thm.\,A]{hsieh-triple}, which also includes an interpolations formula in the cases where $\epsilon_{Q_0}$ is not necessarily trivial. 
\end{proof}


\subsection{CM Hida families}


Let $K_\infty$ be the unique $\Z_p^2$-extension of $K$, and let $K_{\pp^\infty}$ be the maximal subfield of $K_\infty$ unramified outside $\pp$. Put 
\[
\Gamma_\infty:={\rm Gal}(K_\infty/K)\simeq\Z_p^2,\quad\quad
\Gamma_{\pp^\infty}:={\rm Gal}(K_{\pp^\infty}/K)\simeq\Z_p.
\]

For every ideal $\mathfrak{C}\subset\cO_K$ we let $K(\mathfrak{C})$ be the ray class field of $K$ of conductor $\mathfrak{C}$ (so in particular $K_{\pp^\infty}$ is the maximal $\Z_p$-extension of $K$ inside $K(\pp^\infty)$). Let ${\rm Art}_\pp$ be the restriction of the global Artin map to $K_\pp^\times$, with geometric normalisation. Identifying $\Z_p^\times$ and $\cO_{K_\pp}^\times$ via $\iota_p:\overline{\bQ}\hookrightarrow\overline{\bQ}_p$, 
the map ${\rm Art}_\pp$ induces an embedding $1+p\Z_p\rightarrow\Gamma_{\pp^\infty}$. Write $I_\pp^{\rm w}={\rm Art}_\pp(1+p\Z_p)\vert_{K_{\pp^\infty}}$ and let $b\geq 0$ be such that $[\Gamma_{\pp^\infty}:I_\pp^{\rm w}]=p^b$. (Note that $b=0$ if the class number of $K$ is coprime to $p$.) 

Fix a topological generator $\gamma_\pp\in\Gamma_{\pp^\infty}$ with $\gamma_\pp^{p^b}={\rm Art}_\pp(1+p)\vert_{K_{\pp^\infty}}$, and for each variable $S$ let $\Psi_S:\Gamma_\infty\rightarrow\mathscr{O}\dBr{S}^\times$ be the universal character 
given by
\[
\Psi_S(\sigma)=(1+S)^{l(\sigma)},
\]
where $l(\sigma)\in\Z_p$ is such that $\sigma\vert_{K_{\pp^\infty}}=\gamma_\pp^{l(\sigma)}$. Let $\mathbf{v}\in\mathscr{O}$ be such that $\mathbf{v}^{p^b}=1+p$ (after enlarging $\mathscr{O}$ if necessary). For any finite order character $\psi:G_K\rightarrow\mathscr{O}^\times$ of  conductor $\mathfrak{C}$ put
\[
\boldsymbol{\theta}_\psi(S)(q)=\sum_{(\fa,\pp\mathfrak{C})=1}\psi(\sigma_{\fa})\Psi^{-1}_{\mathbf{v}^{-1}(1+S)-1}(\sigma_\fa)q^{\mathbf{N}(\fa)}\in\mathscr{O}\dBr{S}\dBr{q},
\]
where $\sigma_\fa\in{\rm Gal}(K(\mathfrak{C}\pp^\infty)/K)$ is the Artin symbol of $\fa$. 
Then $\boldsymbol{\theta}_\psi(S)$ is a primitive Hida family defined over $\mathscr{O}\dBr{S}$ of tame level $D_K\mathbf{N}(\mathfrak{C})$ and tame character $(\psi\circ\mathscr{V})\eta_{K/\Q}\omega^{-1}$, where 
\[
\mathscr{V}:G_\Q^{\rm ab}\rightarrow G_K^{\rm ab}
\] 
is the transfer map and $\eta_{K/\Q}$ is the quadratic character associated to $K/\Q$.

\subsection{Anticyclotomic $p$-adic $L$-functions}\label{subsec:anticyc-Lp}

Let $f\in S_2(\Gamma_0(pN_f))$, with $p\nmid N_f$, be a $p$-ordinary $p$-stabilised newform of tame level $N_f$ defined over $\mathscr{O}$. Assume that $f$ is the ordinary $p$-stabilisation of the newform $f^\circ\in S_2(\Gamma_0(N_f))$, and let $\alpha_p\in\mathscr{O}^\times$ be the $U_p$-eigenvalue of $f$. Write 
\begin{equation}\label{eq:N-pm}
N_f=N^+ N^-
\end{equation}
with $N^+$ (resp. $N^-$) divisible only by primes which are split (resp. inert) in $K$, and fix an ideal $\mathfrak{N}^+\subset\cO_K$ with $\cO_K/\mathfrak{N}^+\iso\bZ/N^+\bZ$. 


Let $\Gamma^-$ be the Galois group of the anticyclotomic $\bZ_p$-extension $K_\infty^-/K$. By definition, the map $\sigma\mapsto l(\sigma^{1-\cc}\vert_{K_{\pp^\infty}})$ factor through $\Gamma^-$, and we let $\gamma^-$ be the topological generator of $\Gamma^-$ mapping to $1$ under the resulting isomorphism $\Gamma^-\simeq\bZ_p$. We then identity $\mathscr{O}\dBr{\Gamma^-}$ with the power series ring $\mathscr{O}\dBr{T}$ via $\gamma^-\mapsto 1+T$. 

\begin{thm}\label{thm:BD-theta}
Let $\chi$ be a ring class character of $K$ of conductor $c\cO_K$ with values in $\mathscr{O}$, and
assume that: 
\begin{itemize}
	\item[(i)] {} $(pN_f,cD_K)=1$,
    \item[(ii)] {} $N^-$ is the squarefree product of an odd number of primes,
    \item[(iii)] if $q\mid N^-$ is a prime with $q\equiv 1\pmod{p}$, then $\bar{\rho}_f$ is ramified at $q$.
\end{itemize}
Then there exists a unique $\Theta_{f/K,\chi}(T)\in\mathscr{O}\dBr{T}$ such that for every $p$-power root of unity $\zeta$:
\[
\Theta^{}_{f/K,\chi}(\zeta-1)^2=\frac{p^n}{\alpha_p^{2n}}\cdot\mathcal{E}_p(f,\chi,\zeta)^{2}\cdot\frac{L(f^\circ/K\otimes\chi\epsilon_\zeta,1)}{(2\pi)^2\cdot \Omega_{f^\circ,N^-}}\cdot u_K^2\sqrt{D_K}\chi\epsilon_\zeta(\sigma_{\mathfrak{N}^+})\cdot\eps_p,
\]	
where:
\begin{itemize}
\item $n\geq 0$ is such that $\zeta$ has exact order $p^n$, 
\item $\epsilon_\zeta:\Gamma_\infty^-\rightarrow\mu_{p^\infty}$ be the character defined by $\epsilon_\zeta(\gamma^-)=\zeta$,
\item 
$\mathcal{E}_p(f,\chi,\zeta)=
\begin{cases}
(1-\alpha_p^{-1}\chi(\pp))(1-\alpha_p\chi(\overline{\pp}))&
\textrm{if $n=0$,}\\[0.2em] 
1&\textrm{if $n>0$,}
\end{cases}$
\item  $\Omega_{f^\circ,N^-}=4\Vert f^\circ\Vert_{\Gamma_0(N_f)}^2\cdot\eta_{f,N^-}^{-1}$ is the \emph{Gross period} of $f^\circ$ (see \cite[p.~524]{hsieh-triple}),
\item $u_K=\vert\cO_K^\times\vert/2$,
\item $\eps_p\in\{\pm 1\}$ is the local root number of $f^\circ$ at $p$. 
\end{itemize}
\end{thm}

\begin{proof}
See \cite{BDmumford-tate} for the original construction, and  
\cite[Thm.~A]{ChHs1} for the stated interpolation property. 
\end{proof}

When $\chi$ is the trivial character, we write $\Theta_{f/K,\chi}(T)$ simply as $\Theta_{f/K}(T)$.

\subsection{Factorisation of triple product $p$-adic $L$-functions}\label{subsec:factor-L}

Let $f\in S_2(pN_f)$ be a $p$-stabilised newform as in  the preceding section. By Hida theory, $f$ is the specialisation of a unique primitive Hida family $\bff\in S^o(N_f,\cR)$ 
at an arithmetic point $Q_0\in\mathfrak{X}_\cR^+$ of weight $2$.  
Let $\ell\nmid pN_f$ be a prime split in $K$, and let $\chi$ be a ring class character of $K$ of conductor $\ell^m\cO_K$ for some  $m>0$. Denoting by $\cc$ the non-trivial automorphism of $K/\bQ$, write $\chi=\psi/\psi^\cc$ with $\psi$ a ray class character modulo $\ell^m\cO_K$ with $\psi^\cc(\sigma)=\psi(\cc\sigma\cc^{-1})$. Let
\begin{equation}\label{eq:gg*}
\bfg=\boldsymbol{\theta}_{\psi}(S_1)\in\mathscr{O}\dBr{S_1}\dBr{q}, 
\quad
\bfg^*=\boldsymbol{\theta}_{\psi^{-1}}(S_2)\in\mathscr{O}\dBr{S_2}\dBr{q}
\end{equation}
be the primitive CM Hida families (of tame level $C=D_K\ell^{2m}$)  attached to $\psi$ and $\psi^{-1}$, respectively.

The triple $(\bff,\bfg,\bfg^*)$ satisfies conditions (\ref{eq:a}) and (\ref{eq:+1}) and the associated triple product $p$-adic $L$-function $\mathscr{L}_p^{\bff}(\bff,\bfg,\bfg^*)$ of Theorem~\ref{thm:hsieh-triple} is an element in $\mathcal{R}=\cR\hat\otimes_{\mathscr{O}}\mathscr{O}\dBr{S_1}\hat\otimes_{\mathscr{O}}\mathscr{O}\dBr{S_2}\simeq\cR\dBr{S_1,S_2}$. Put $S=S_1$. In the following, we let 
\begin{equation}\label{eq:Lp-1}
\mathscr{L}_p^\bff(f,\ggstar)\in\mathscr{O}\dBr{S}
\end{equation}
be the image of $\mathscr{L}_p^\unb(\bff,\bfg,\bfg^*)$ under the map $\cR\dBr{S_1,S_2}\rightarrow\mathscr{O}\dBr{S_1,S_2}$ given by the specialisation $Q_0:\cR\rightarrow\mathscr{O}$ composed with the quotient  $\mathscr{O}\dBr{S_1,S_2}\rightarrow\mathscr{O}\dBr{S_1,S_2}/(S_1-S_2)$. 

Let $K(\chi)$ be the field obtained by adjoining to $K$ the values of $\chi$, and put $K(\chi,\alpha_p)=K(\chi)(\alpha_p)$.

\begin{prop}\label{prop:factor-L}
Assume that: 
\begin{itemize}
\item[(i)] $N^-$ is the squarefree product of an odd number of primes, 
\item[(ii)] if $q\mid N^-$ is a prime with $q\equiv 1\pmod{p}$, then $\bar{\rho}_f$ is ramified at $q$. 
\end{itemize}
Set $T=\mathbf{v}^{-1}(1+S)-1$. Then 
\[
\mathscr{L}_p^\unb(f,\ggstar)(S)=
\pm\mathbf{w}^{-1}\cdot\Theta_{f/K}(T)\cdot C_{f,\chi}\cdot \sqrt{L^{\rm alg}(f/K\otimes\chi,1)}\cdot\frac{\eta_{f^\circ}}{\eta_{f^\circ,N^-}},
\]
where $\mathbf{w}$ is a unit in $\mathscr{O}\dBr{T}$, $C_{f,\chi}\in K(\chi,\al_p)^\x$, and  
\[
L^{\rm alg}(f/K\otimes\chi,1):=\frac{L(f/K\otimes\chi,1)}{4\pi^2\lVert f^\circ\rVert_{\Gamma_0(N_{f})}^2}\in K(\chi).
\] 
\end{prop}

\begin{proof} 	
This is immediate from \cite[Prop.~8.1]{hsieh-triple} and the interpolation property 
of $\Theta_{f/K,\chi}(0)$. 
\end{proof}





\section{Selmer group decompositions}

In this section we introduce two different Selmer groups associated to triple products of modular forms, and relate them 
to anticyclotomic Selmer groups attached to a single modular form.

\subsection{Selmer groups for triple products of modular form}

Let $(\bff,\bfg,\bfh)$ be a triple of primitive Hida families as in $\S\ref{subsubsec:triple-hida}$ satisfying (\ref{eq:a}), and let $\Vdag=\mathbf{V}\otimes\mathcal{X}^{-1}$ be the self-dual twist of the associated 
big  Galois representation.

\begin{defn}\label{def:local-p}
Put
\[
\mathscr{F}^{\rm bal}_p(\Vdag)=\mathscr{F}_p^2(\Vdag)
:=\bigl(V_{\bff}^+\otimes V_{\bfg}^+\otimes V_{\bfh}+V_{\bff}^+\otimes V_{\bfg}\otimes V_{\bfh}^++V_{\bff}\otimes V_{\bfg}^+\otimes V_{\bfh}^+\bigr)\otimes\mathcal{X}^{-1},
\]
and define the \emph{balanced local condition} $\rH^1_{\rm bal}(\Q_p,\Vdag)$ 
by 
\[
\rH^1_{\rm bal}(\Q_p,\Vdag):={\rm im}\bigl\{\rH^1(\Q_p,\mathscr{F}_p^{\rm bal}(\Vdag))\rightarrow\rH^1(\Q_p,\Vdag)\bigr\}.
\]
Similarly, put 
\[
\mathscr{F}_p^{\unb}(\Vdag):=\bigl(V_{\bff}^+\otimes V_{\bfg}\otimes V_{\bfh}\bigr)\otimes\mathcal{X}^{-1},
\] 
and define the \emph{$\bff$-unbalanced local condition} $\rH^1_{\unb}(\Q_p,\Vdag)$ 
by
\[
\rH^1_{\bff}(\Q_p,\Vdag):={\rm im}\bigl\{\rH^1(\Q_p,\mathscr{F}_p^{\bff}(\Vdag))\rightarrow\rH^1(\Q_p,\Vdag)\bigr\}.
\]
\end{defn}

It is easy to see that the maps appearing in these definitions are injective, and in the following we shall use this to identify $\rH^1_{\any}(\Q_p,\Vdag)$ with $\rH^1(\Q_p,\mathscr{F}_p^\any(\Vdag))$ for $\any\in\{{\rm bal},\unb\}$. Fix a finite set of primes $S$ containing $\infty$ and the primes dividing $N_fN_gN_g$, and let $G_{\bQ,S}$ be the Galois group of the maximal extension of $\Q$ unramified outside $S$.

\begin{defn}\label{def:Sel-bu}
Let $\any\in\{{\rm bal},\unb\}$, and define the Selmer group ${\rm Sel}^\any(\Q,\Vdag)$ by
\[
{\rm Sel}^\any(\Q,\Vdag):=\ker\biggl\{\rH^1(G_{\Q,S},\Vdag)\rightarrow\frac{\rH^1(\Q_p,\Vdag)}{\rH^1_\any(\Q_p,\Vdag)}\biggr\}
\]
We call ${\rm Sel}^{\rm bal}(\Q,\Vdag)$ (resp. ${\rm Sel}^{\bff}(\Q,\Vdag)$) the \emph{balanced} (resp. \emph{$\bff$-unbalanced}) Selmer group.
\end{defn}

Let $\Adag={\rm Hom}_{\Z_p}(\Vdag,\mu_{p^\infty})$ be the arithmetic dual of $\Vdag$, and for $\any\in\{{\rm bal},\bff\}$ define $\rH^1_{\any}(\Q_p,\Adag)\subset\rH^1(\Q_p,\Adag)$ to be the orthogonal complement of $\rH^1_\any(\Q_p,\Vdag)$ under the local Tate duality
\[
\rH^1(\Q_p,\Vdag)\times\rH^1(\Q_p,\Adag)\rightarrow\Q_p/\Z_p.
\]
Similarly as above, we then define the balanced and $\bff$-unbalanced Selmer groups with coefficients in $\Adag$ by
\[
{\rm Sel}^\any(\Q,\Adag):=\ker\biggl\{\rH^1(G_{\Q,S},\Adag)\rightarrow\frac{\rH^1(\Q_p,\Adag)}{\rH_\any^1(\Q_p,\Adag)}\times\prod_{\ell\in S\smallsetminus\{p\}}\rH^1(\Q_\ell,\Adag)\biggr\},
\]
and let $X^\any(\Q,\Adag)={\rm Hom}_{\Z_p}({\rm Sel}^\any(\Q,\Adag),\Q_p/\Z_p)$ denote the Pontryagin dual of ${\rm Sel}^\any(\Q,\Adag)$.

\subsection{Anticyclotomic Selmer groups for modular forms}\label{subsec:Sel-f}

Let $f\in S_2(\Gamma_0(pN_f))$ be an ordinary $p$-stabilised newform and $K/\bQ$ an imaginary quadratic field as in $\S\ref{subsec:anticyc-Lp}$. 

Let $V_f$ be the $p$-adic Galois representation associated to $f$. We adopt the convention that if $f$ corresponds to the isogeny class of an elliptic curve $E/\bQ$, then $V_f\simeq V_pE$ (rather than its dual). 
By $p$-ordinarity, restricted to $G_{\Q_p}$ the representation $V_f$ fits into a short exact sequence
\[
0\rightarrow V_f^+\rightarrow V_f\rightarrow V_f^-\rightarrow 0
\]
with $V_f^\pm$ both $1$-dimensional, and with the $G_{\Q_p}$-action on $V_f^-$ given by the unramified character sending arithmetic Frobenius to $\alpha_p\in\mathscr{O}^\times$, the $U_p$-eigenvalue of $f$.

Below we fix $\Sigma$ to be any finite set of places of $K$ containing $\infty$ and the prime dividing $pN_f$, and for any field extension $L/K$ let $G_{L,\Sigma}$ the Galois group of the maximal extension of $L$ unramified outside $\Sigma$.

\begin{defn}\label{def:Sel-f}
Let $L$ be a finite field extension of $K$, and let $\mathscr{F}=\{\mathscr{F}_v(V_f)\}_{v\mid p}$ be a collection a $G_{K_v}$-stable subpaces $\mathscr{F}_v(V_f)\subset V_f$ for $v\mid p$. We define the  \emph{Greenberg Selmer group} ${\rm Sel}_{\mathscr{F}}(L,V_f)$ by
\[
{\rm Sel}_{\mathscr{F}}(L,V_f):={\rm ker}\biggl\{\rH^1(G_{L,\Sigma},V_f)\rightarrow\prod_{w}\frac{\rH^1(L_w,V_f)}{\rH^1_{\mathscr{F}}(L_w,V_f)}\biggr\},
\]
where $w$ runs over the finite primes of $L$ lying above a prime $v\in\Sigma$, and 
\[
\rH^1_{\mathscr{F}}(L_w,V_f):=\begin{cases}
\ker\{\rH^1(L_w,V_f)\rightarrow\rH^1(L_w^{\rm ur},V_f)\}&\textrm{if $w\nmid p$,}\\[0.2em]
{\rm im}\{\rH^1(L_w,\mathscr{F}_v(V_f))\rightarrow\rH^1(L_w,V_f)\}&\textrm{if $w\mid v\mid p$}.
\end{cases}
\]
\end{defn}

We shall be particularly interested in the following choices of $\mathscr{F}$:
\begin{itemize}
\item The \emph{relaxed-strict} Selmer group ${\rm Sel}_{\emptyset,0}(L,V_f)$ obtained by taking
\[
\mathscr{F}_v(V_f)=\begin{cases}
V_f&\textrm{if $v=\pp$,}\\[0.2em]
0&\textrm{if $v=\ppbar$.}
\end{cases}
\]
\item The \emph{ordinary} Selmer group ${\rm Sel}(L,V_f)$ obtained by taking
$\mathscr{F}_v(V_f)=V_f^+$ for all $v\mid p$.
\end{itemize}

For a $G_K$-stable lattice $T_f\subset V_f$, we let $\rH^1_{\mathscr{F}}(L_w,T_f)$ be the inverse image of $\rH^1_{\mathscr{F}}(L_w,V_f)$ under the natural map $\rH^1(L_w,T_f)\rightarrow\rH^1(L_w,V_f)$, and define ${\rm Sel}_{\mathscr{F}}(L,T_f)$ by the same recipe as above. Then, for $A_f:={\rm Hom}_{\Z_p}(T_f,\mu_{p^\infty})$, we define the \emph{dual Selmer group} ${\rm Sel}_{\mathscr{F}^*}(L,A_f)$ by
\[
{\rm Sel}_{\mathscr{F}^*}(L,A_f):=\ker\biggl\{\rH^1(G_{L,\Sigma},A_f)\rightarrow\prod_w\frac{\rH^1(L_w,A_f)}{\rH^1_{\mathscr{F}^*}(L_w,A_f)}\biggr\}
\]
where $\rH^1_{\mathscr{F}^*}(L_w,A_f)$ is the orthogonal complement of $\rH^1_{\mathscr{F}}(L_w,T_f)$ under local Tate duality
\[
\rH^1(L_w,T_f)\times\rH^1(L_w,A_f)\rightarrow\bQ_p/\bZ_p.
\]

Write $K_m^-$ for the subextension of the the anticyclotomic $\Z_p$-extension $K_\infty^-$ with $[K_m^-:K]=p^m$, and put
\[
{\rm Sel}_{\mathscr{F}}(K_\infty^-,T_f):=\varprojlim_m{\rm Sel}_{\mathscr{F}}(K_m^-,T_f),\quad
{\rm Sel}_{\mathscr{F}}(K_\infty^-,A_f):=\varinjlim_m{\rm Sel}_{\mathscr{F}}(K_m^-,A_f),
\] 
where the limits are with respect to the corestriction and restriction maps, respectively. Let 
$\Psi:G_K\rightarrow\mathscr{O}\dBr{\Gamma^-}^\times$ be the character arising from the projection $G_K\rightarrow\Gamma^-$. Writing $T_f\otimes\Psi^{-1}$ (resp. $A_f\otimes\Psi$) for the module $T_f\otimes_{\mathscr{O}}\mathscr{O}\dBr{\Gamma^-}$ (resp. $A_f\otimes_{\mathscr{O}}\mathscr{O}\dBr{\Gamma^-}$) with $G_K$-action on the second factor given by $\Psi^{-1}$ (resp. $\Psi$), we then have natural $\mathscr{O}\dBr{\Gamma^-}$-module pseudo-isomorphisms
\begin{equation}\label{eq:Sh-Sel}
{\rm Sel}_{\mathscr{F}}(K_\infty^-,T_f)\sim{\rm Sel}_{\mathscr{F}}(K,T_f\otimes\Psi^{-1}),\quad
{\rm Sel}_{\mathscr{F}}(K_\infty^-,A_f)\sim{\rm Sel}_{\mathscr{F}}(K,A_f\otimes\Psi),
\end{equation}
where the Selmer groups on the right-hand side are defined is the same way as in Definition~\ref{def:Sel-f}, with $\mathscr{F}_v(T_f\otimes\Psi^{-1}):=(\mathscr{F}_v(V_f)\cap T_f)\otimes\Psi^{-1}$ and $\mathscr{F}_v(A_f\otimes\Psi):={\rm Hom}_{\Z_p}(T_f/(\mathscr{F}_v(V_f)\cap T_f),\mu_{p^\infty})\otimes\Psi$.

\subsection{Decomposition of triple product Selmer groups}\label{subsec:dec-S}

Suppose now that $\bff$ is the Hida  family passing through the $p$-stabilised newform $f\in S_2(\Gamma_0(pN_f))$, and $(\bfg,\bfg^*)=(\boldsymbol{\theta}_\psi(S_1),\boldsymbol{\theta}_{\psi^{-1}}(S_2))$ are CM Hida families as in (\ref{eq:gg*}). 

For any primitive Hida family $\boldsymbol{\phi}$, we now let $V_{\boldsymbol{\phi}}$ be the realization of $\rho_{\boldsymbol{\phi}}$ arising in the $p$-adic \'{e}tale cohomology of the $p$-tower of modular curves as in Ohta's works \cite{ohta-etI,ohta-etII}, following the conventions in \cite[\S{7.2}]{KLZ2} (except that a ``Hida family'' for us is a ``branch''  
in their sense). 

In the case of $\bfg$ (and similarly $\bfg^*$), when $\chi:=\psi/\psi^\cc$ satisfies 
\begin{equation}\label{eq:chi-dist}
\bar{\chi}\vert_{G_{K_\pp}}\neq 1,
\end{equation}
where $G_{K_\pp}\subset G_K$ is a decomposition group at $\pp$, it follows from a slight extension of the isomorphism in \cite[Cor.\,5.2.5]{LLZ-K} (see \cite[$\S${3.2.3}]{BL-ord}) that
\[
V_{\bfg}\simeq{\rm Ind}_K^\Q(\psi^{-1}\Psi_{T_1}),
\]
where $T_1=\mathbf{v}^{-1}(1+S_1)-1$. Suppose $Q_0:\cR\rightarrow\mathscr{O}$ is such that  $f$ is the specialisation of $\bff$ at $Q_0$, and write $\Vsdag$ for the resulting specialisation of $\Vdag$. Similarly setting $T_2=\mathbf{v}^{-1}(1+S_2)-1$ and noting that $(\det V_\bfg)(\det V_{\bfg^*})=\Psi_{T_1}\Psi_{T_2}\circ\mathscr{V}$, we then have
\begin{equation}\label{eq:dec-V}
\begin{aligned}
\Vsdag&\simeq T_f\otimes{\rm Ind}_K^\Q(\psi^{-1}\Psi_{T_1})\otimes{\rm Ind}_K^\Q(\psi\Psi_{T_2})\otimes(\Psi_{T_1}^{1/2}\Psi_{T_2}^{1/2}\circ\mathscr{V})^{-1}\\
&\simeq\bigl(T_f\otimes{\rm Ind}_{K}^\Q\Psi_{W_1}^{1-\cc}\bigr)\oplus\bigl(T_f\otimes{\rm Ind}_K^\Q\chi^{-1}\Psi_{W_2}^{1-\cc}\bigr),
\end{aligned}
\end{equation}
where $T_f$ is the specialisation of $V_{\bff}$ at $Q_0$ and
\[
W_1=\mathbf{v}^{-1}(1+S_1)^{1/2}(1+S_2)^{1/2}-1,\quad W_2=(1+S_1)^{1/2}(1+S_2)^{-1/2}-1,
\] 
In particular, together with Shapiro's lemma this gives
\begin{equation}\label{eq:shapiro}
\rH^1(\Q,\Vsdag)\simeq\rH^1(K,T_f\otimes\Psi_{W_1}^{1-\cc})\oplus\rH^1(K,T_f\otimes\chi^{-1}\Psi_{W_2}^{1-\cc})
\end{equation}

\begin{prop}\label{prop:factor-S}
Suppose $\psi$ is a ray class character of $K$ such that $\chi=\psi/\psi^\cc$ satisfies \eqref{eq:chi-dist}. Then under the isomorphism (\ref{eq:shapiro}), the balanced Selmer group 
decomposes as
\begin{align*}
{\rm Sel}^{\rm bal}(\Q,\Vsdag)&\simeq{\rm Sel}_{\emptyset,0}(K,T_f\otimes\Psi_{W_1}^{1-\cc})\oplus{\rm Sel}(K,T_f\otimes\chi^{-1}\Psi_{W_2}^{1-\cc}),
\end{align*}
and the $\unb$-unbalanced Selmer group 
decomposes as
\begin{align*}
{\rm Sel}^{\unb}(\Q,\Vsdag)&\simeq
{\rm Sel}(K,T_f\otimes\Psi_{W_1}^{1-\cc})\oplus{\rm Sel}(K,T_f\otimes\chi^{-1}\Psi_{W_2}^{1-\cc}).\end{align*}
\end{prop}

\begin{proof}
This is shown in \cite[Prop.~5.3.1]{C-Do}; we recall the details for the convenience of the reader. From (\ref{eq:dec-V}), we see that 
\[
\Vsdag\vert_{G_{\Q_p}}\simeq\bigl(T_f\otimes\Psi_{W_1}^{1-\cc}\bigr)\oplus(T_f\otimes\Psi_{W_1}^{\cc-1}\bigr)\oplus\bigl(T_f\otimes\chi^{-1}\Psi_{W_2}^{1-\cc}\bigr)\oplus\bigl(T_f\otimes\chi\Psi_{W_2}^{\cc-1}\bigr)
\]
noting that $\chi^\cc=\chi^{-1}$, and we find that the balanced local condition is given by
\begin{equation}\label{eq:bal-CM}
\mathscr{F}^{\rm bal}_p(\Vsdag)
\simeq\bigl(T_f\otimes\Psi_{W_1}^{1-\cc}\bigr)\oplus\bigl(T_f^+\otimes\chi^{-1}\Psi_{W_1}^{1-\cc}\bigr)\oplus\bigl(T_f^+\otimes\chi\Psi_{W_2}^{\cc-1}\bigr),
\end{equation}
where $T_f^+\subset T_f$ is the specialisation of $V_\bff^+$ at $Q_0$. Put $\widetilde{\mathbb{V}}_{f\boldsymbol{gg}^*}^\dagger=(T_f\otimes\Psi_{W_1}^{1-\cc})\oplus(T_f\otimes\chi^{-1}\Psi_{W_2}^{1-\cc})$, and note that
\begin{equation}\label{eq:sh-iso}
\rH^1(\Q,\Vsdag)\simeq\rH^1(K,\widetilde{\mathbb{V}}_{f\boldsymbol{gg}^*}^\dagger)
\end{equation}
by Shapiro's lemma. For $\qq\in\{\pp,\ppbar\}$, letting $\mathscr{F}_\qq^{\rm bal}(\widetilde{\mathbb{V}}_{f\boldsymbol{gg}^*}^\dagger)$ be the submodule of $\widetilde{\mathbb{V}}_{f\boldsymbol{gg}^*}^\dagger$ corresponding to $
\mathscr{F}^{\rm bal}_p(\Vsdag)$, from \eqref{eq:bal-CM} we have
\begin{equation}\label{eq:bal-shapiro}
\begin{cases}
\mathscr{F}^{\rm bal}_\pp(\widetilde{\mathbb{V}}_{f\boldsymbol{gg}^*}^\dagger)\simeq\bigl(T_f\otimes\Psi_{W_1}^{1-\cc}\bigr)\oplus\bigl(T_f^+\otimes\chi^{-1}\Psi_{W_2}^{1-\cc}\bigr),\\[0.2em]
\mathscr{F}^{\rm bal}_{\ppbar}(\widetilde{\mathbb{V}}_{f\boldsymbol{gg}^*}^\dagger)\simeq T_f^+\otimes\chi^{-1}\Psi_{W_2}^{1-\cc},
\end{cases}
\end{equation}
using again that $\chi^\cc=\chi^{-1}$ for the second isomorphism. As a result, 
under the isomorphism \eqref{eq:sh-iso} we have
\begin{align*}
{\rm Sel}^{\rm bal}(\Q,\Vsdag)&\simeq\\
\quad\ker\biggl\{
\rH^1(G_{K,\Sigma},&\widetilde{\mathbb{V}}_{f\boldsymbol{gg}^*}^\dagger)\rightarrow\frac{\rH^1(K_\pp,\widetilde{\mathbb{V}}_{f\boldsymbol{gg}^*}^\dagger)}{\rH^1(K_\pp,(T_f\otimes\Psi_{W_1}^{1-\cc})\oplus(T_f^+\otimes\chi^{-1}\Psi_{W_2}^{1-\cc}))}\times\frac{\rH^1(K_{\ppbar},\widetilde{\mathbb{V}}_{f\boldsymbol{gg}^*}^\dagger)}{\rH^1(K_{\ppbar},T_f^+\otimes\chi^{-1}\Psi_{W_2}^{1-\cc})}\bigr\},
\end{align*}
yielding the result in this case. 
Similarly, we find that the $\unb$-balanced local condition is given by
\begin{equation}
\begin{cases}
\mathscr{F}^{\unb}_\pp(\widetilde{\mathbb{V}}_{f\boldsymbol{gg}^*}^\dagger)\simeq\bigl(T_f^+\otimes\Psi_{W_1}^{1-\cc}\bigr)\oplus\bigl(T_f^+\otimes\chi^{-1}\Psi_{W_2}^{1-\cc}\bigr),\\[0.2em]
\mathscr{F}^{\unb}_{\ppbar}(\widetilde{\mathbb{V}}_{f\boldsymbol{gg}^*}^\dagger)\simeq\bigl(T_f^+\otimes\Psi_{W_1}^{1-\cc}\bigr)\oplus\bigl(T_f^+\otimes\chi^{-1}\Psi_{W_2}^{1-\cc}\bigr),\nonumber
\end{cases}
\end{equation}
and as above we arrive at the claimed description of ${\rm Sel}^\unb(\Q,\Vsdag)$.
\end{proof}

As a consequence, we also obtain decompositions for corresponding Selmer groups with coefficients in $\Asdag={\rm Hom}_{\Z_p}(\Vsdag,\mu_{p^\infty})$, mirroring in the case of ${\rm Sel}^{\unb}(\Q,\Asdag)$ the factorisation of $p$-adic $L$-functions in Proposition~\ref{prop:factor-L}. Put $A_f={\rm Hom}_{\Z_p}(T_f,\mu_{p^\infty})$.

\begin{cor}\label{cor:factor-S}
Under the asumption in Proposition~\ref{prop:factor-S}, we have the decompositions
\begin{align*}
{\rm Sel}^{\rm bal}(\Q,\Asdag)&\simeq{\rm Sel}_{0,\emptyset}(K,A_f\otimes\Psi_{W_1}^{\cc-1})\oplus{\rm Sel}(K,A_f\otimes\chi\Psi_{W_2}^{\cc-1}),\\
{\rm Sel}^{\unb}(\Q,\Asdag)&\simeq
{\rm Sel}(K,A_f\otimes\Psi_{W_1}^{\cc-1})\oplus{\rm Sel}(K,A_f\otimes\chi\Psi_{W_2}^{\cc-1}).
\end{align*}
\end{cor}

\begin{proof}
This is immediate from Proposition~\ref{prop:factor-S} and local Tate duality.
\end{proof}

\section{Iwasawa main conjectures}\label{sec:IMC}

Keeping the setting from  $\S\ref{subsec:dec-S}$, and put $\Lambda_2=\mathscr{O}\dBr{S_1,S_2}$ and
\[
\Vsdagg=\Vsdag\hat\otimes_{\Lambda_2}\Lambda_2/(S_1-S_2),
\]
where $\Lambda_2=\mathscr{O}\dBr{S_1,S_2}$. 
In this section we recall the diagonal cycle main conjecture formulated in \cite{ACR} specialised to our setting, and explain its relation with the anticyclotomic Iwasawa main conjecture for $f$.

\subsection{Diagonal cycle main conjecture}

Let $\kappa(\bff,\bfg,\bfg^*)\in\rH^1(\Q,\Vdag)$ be the \emph{big diagonal class} constructed in \cite[\S{8.1}]{BSV}, where $\bff$ is the primitive Hida family passing through $f$, and denote by 
\begin{equation}\label{eq:diag}
\kappa(f,\ggstar)\in\rH^1(\Q,\Vsdagg)
\end{equation}
its specialisation. More precisely, from \emph{loc.\,cit.} one  directly obtains a class as above with coefficient in a representation $\Vdag(N)$  isomorphic (non-canonically) to finitely many copies of $\Vdag$, where $N={\rm lcm}(N_f,N_g,N_{g^*})={\rm lcm}(N_f,D_K\ell)$; to obtain \eqref{eq:diag} we apply the projection $\Vdag(N)\rightarrow\Vdag$ corresponding to the choice of level-$N$ test vectors for $(\bff,\bfg,\bfg^*)$ furnished by  \cite[Thm.~A]{hsieh-triple}.

From the decompositions in Proposition~\ref{prop:factor-S} we obtain
\begin{equation}\label{eq:dec-S-gg}
\begin{aligned}
{\rm Sel}^{\rm bal}(\Q,\Vsdagg)&\simeq{\rm Sel}_{\emptyset,0}(K,T_f\otimes\Psi_{T}^{1-\cc})\oplus{\rm Sel}(K,T_f\otimes\chi^{-1}\Psi_{W_2}^{1-\cc})_{/W_2},\\
{\rm Sel}^{\unb}(\Q,\Vsdagg)&\simeq
{\rm Sel}(K,T_f\otimes\Psi_{T}^{1-\cc})\oplus{\rm Sel}(K,T_f\otimes\chi^{-1}\Psi_{W_2}^{1-\cc})_{/W_2},
\end{aligned}
\end{equation}
where $T=T_1=\mathbf{v}^{-1}(1+S_1)-1$ and the subscript $/W_2$ denotes the cokernel of multiplication by $W_2$. Put
\[
\mathscr{F}_p^3(\Vsdag):=T_f^+\hat\otimes_{\mathscr{O}}V_{\bff}^+\hat\otimes_{\mathscr{O}}V_{\bfg^*}^+\otimes\mathcal{X}^{-1},
\]
and denote by $\mathscr{F}_p^3(\Vsdagg)$ the resulting submodule of $\Vsdagg$. Similarly defining $\mathscr{F}_p^{\rm bal}(\Vsdagg)$ from $\mathscr{F}_p^{\rm bal}(\Vsdag)=\mathscr{F}_p^2(\Vsdag)$, in terms of the description in \eqref{eq:bal-shapiro} we find that
\begin{equation}\label{eq:fil-3}
\begin{cases}
\mathscr{F}^{\rm bal}_\pp(\widetilde{\mathbb{V}}_{f\ggstar}^\dagger)/\mathscr{F}^{3}_\pp(\widetilde{\mathbb{V}}_{f\ggstar}^\dagger)\simeq\bigl(T_f^-\otimes\Psi_{T}^{1-\cc}\bigr)\oplus\bigl(T_f^+\otimes\chi^{-1}\Psi_{W_2}^{1-\cc}\bigr)_{/W_2},\\[0.2em]
\mathscr{F}^{\rm bal}_{\ppbar}(\widetilde{\mathbb{V}}_{f\ggstar}^\dagger)/\mathscr{F}^{3}_{\ppbar}(\widetilde{\mathbb{V}}_{f\ggstar}^\dagger)\simeq( T_f^+\otimes\chi^{-1}\Psi_{W_2}^{1-\cc})_{/W_2},
\end{cases}
\end{equation}
where $T_f^-:=T_f/T_f^+$. 

It follows from \cite[Cor.~8.2]{BSV} that $\kappa(f,\ggstar)$ lies in the balanced Selmer group ${\rm Sel}^{\rm bal}(\Q,\Vsdagg)$. Therefore, viewing this class in $\rH^1(K,\widetilde{\mathbb{V}}^\dagger_{f\ggstar})$ via the isomorphism \eqref{eq:sh-iso}, we can consider the image of ${\rm loc}_\pp(\kappa(f,\ggstar))$ under the natural map
\[
p_f^-:
\rH^1\bigl(K_\pp,\mathscr{F}^{\rm bal}_\pp(\widetilde{\mathbb{V}}_{f\ggstar}^\dagger)\bigr)
\rightarrow\rH^1\bigl(K_\pp,\mathscr{F}^{\rm bal}_\pp(\widetilde{\mathbb{V}}_{f\ggstar}^\dagger)/\mathscr{F}^{3}_\pp(\widetilde{\mathbb{V}}_{f\ggstar}^\dagger)\bigr)\rightarrow\rH^1(K_\pp,T_f^-\otimes\Psi_T^{1-\cc})
\]
arising from the projection onto the first factor in \eqref{eq:fil-3}. 

Put
\[
\Lambda=\mathscr{O}\dBr{T}.
\]
Let $u:=1+p$, and for any $\Lambda$-module $M$ and integer $k$, denote by $M_k$ the specialisation of $M$ at $T=u^{k-2}-1$. Letting $\Phi$ be the field of fractions of $\mathscr{O}$, it is then easy to see that the Bloch--Kato logarithm and dual exponential maps define isomorphisms
\begin{equation}\label{eq:BK}
\begin{aligned}
{\rm log}_{\pp}:\rH^1(K_{\pp},T_f^-\otimes\Psi_T^{1-\cc})_k\otimes\Q_p&\rightarrow \Phi,\quad k\geq 3,\\
{\rm exp}_{\pp}^*:\rH^1(K_{\pp},T_f^-\otimes\Psi_T^{1-\cc})_k\otimes\Q_p&\rightarrow \Phi,\quad k=2.
\end{aligned}
\end{equation}


The following fundamental result due to Bertolini--Seveso--Venerucci  and Darmon--Rotger relates $p_f^-({\rm loc}_\pp(\kappa(f,\ggstar)))$ to the restricted triple product $p$-adic $L$-function $\mathscr{L}_p^{\unb}(f,\ggstar)$ in \eqref{eq:Lp-1}. 

\begin{thm}[Explicit reciprocity law]\label{thm:ERL}
There is an injective $\Lambda$-module homomorphism
\[
{\rm Log}^{\unb}:\rH^1(K_\pp,T_f^-\otimes\Psi_{T}^{1-\cc})\rightarrow\Lambda
\]
with pseudo-null cokernel satisfying for any $\mathfrak{Z}\in\rH^1(K_\pp,T_f^-\otimes\Psi_{T}^{1-\cc})$ the interpolation property
\[
{\rm Log}^{\unb}(\mathfrak{Z})_k=c_k\times\begin{cases}
{\rm log}_{\pp}(\mathfrak{Z}_k)&\textrm{if $k\geq 3$,}\\[0.2em]
{\rm exp}_{\pp}(\mathfrak{Z}_k)&\textrm{if $k=2$,}
\end{cases}
\]
where $c_k$ is an explicit nonzero constant, and such that
\[
{\rm Log}^{\unb}\bigl(p_f^-({\rm loc}_\pp(\kappa(f,\ggstar)))\bigr)=\mathscr{L}_p^\unb(f,\ggstar).
\]
\end{thm}

\begin{proof}
The map ${\rm Log}^\bff$ is obtained by specialising the three-variable big regulator map of \cite[\S{7.1}]{BSV}, and the explicit reciprocity law is a consequence of the results of \cite[Thm.\,A]{BSV} and  \cite[Thm.\,10]{DR3}. The details are explained in \cite[Prop.\,7.3]{ACR} and \cite[Thm.\,5.1.1]{C-Do}.
\end{proof}


The next result shows the equivalence between a ``diagonal cycle main conjecture'' in the spirit of \cite[Conj.\,B]{PR-HP} and the Iwasawa--Greenberg main conjecture for $\mathscr{L}_p^{\unb}(f,\ggstar)$.

\begin{prop}\label{prop:equiv}
The following statements (1)-(2) are equivalent:
\begin{enumerate}
\item{} $\kappa(f,\ggstar)$ is not $\Lambda$-torsion, 
\[
{\rm rank}_\Lambda\bigl({\rm Sel}^{\rm bal}(\Q,\Vsdagg)\bigr)={\rm rank}_\Lambda\bigl(X^{\rm bal}(\Q,\Asdagg)\bigr)=1, 
\]
and the following equality holds in $\Lambda\otimes\Q_p$:
\[
{\rm char}_{\Lambda}\bigl(X^{\rm bal}(\Q,\Asdagg)_{\rm tors}\bigr)={\rm char}_{\Lambda}\bigl({\rm Sel}^{\rm bal}(\Q,\Vsdagg)/(\kappa(f,\ggstar))\bigr)^2.
\]
where the subscript ${\rm tors}$ denotes the $\Lambda$-torsion submodule.
\item{} $\mathscr{L}_p^\unb(f,\ggstar)$ is nonzero, the modules ${\rm Sel}^{\unb}(\Q,\Vsdagg)$ and $X^{\unb}(\Q,\Asdagg)$ are both $\Lambda$-torsion, and 
\[
{\rm char}_\Lambda\bigl(X^{\unb}(\Q,\Asdagg)\bigr)=\bigl(\mathscr{L}_p^\unb(f,\ggstar)^2\bigr)
\]
in $\Lambda\otimes\Q_p$.
\end{enumerate}
\end{prop}

\begin{proof}
This follows from an argument by now standard using Theorem~\ref{thm:ERL} and global duality as in \cite[Thm.\,7.15]{ACR}. 
\end{proof}



\subsection{Anticyclotomic main conjecture for modular forms}

For our later use, here we record some known results on the anticyclotomic Iwasawa main conjecture for modular forms. 

Let $\mathbb{F}$ be the residue field of $\mathscr{O}$, and denote by 
\[
\bar{\rho}_f:G_\Q\rightarrow{\rm GL}_2(\mathbb{F})
\] 
the residual representation associated to $f$. 
Factor $N_f=N^+N^-$ as in \eqref{eq:N-pm}. Following \cite{KPW}, we say that $\emph{Assumption~(A)}$ holds if:
\begin{itemize}
\item $\bar{\rho}_f$ is absolutely irreducible,
\item $\bar{\rho}_f$ is ramified at every prime $q\mid N^-$ with $q\equiv\pm{1}\pmod{p}$,
\item either $p>5$ or the image of $\bar{\rho}_f$ contains a conjugate of ${\rm GL}_2(\mathbb{F}_p)$.
\end{itemize}
Note that this is the same as Assumption~(A) from \cite{KPW}, except that the latter also includes the non-anomalous condition $a_p(f)\not\equiv{1}\pmod{p}$; that the next result holds without this condition follows from recent advances on Ihara's lemma, \cite{manning-shotton}.  

\begin{thm}\label{thm:BD-SU}
Suppose that the ordinary $p$-stabilised newform $f\in S_2(\Gamma_0(pN_f))$ is $p$-old, and that $N^-$ is the squarefree product of an odd number of primes. Under Assumption~(A),
the module ${\rm Sel}(K,A_f\otimes\Psi_T^{\cc-1})$ is $\Lambda$-cotorsion, and
\[
{\rm char}_\Lambda\bigl({\rm Sel}(K,A_f\otimes\Psi_T^{\cc-1})^\vee\bigr)=\bigl(\Theta_{f/K}(T)^2\bigr)
\]
in $\Lambda\otimes\Q_p$.
\end{thm}

\begin{proof}
This follows from the combination of the main result of Bertolini--Darmon \cite{bdIMC}, as later refined by several authors (see \cite{pollack-weston,ChHs2,KPW}), showing the $\Lambda$-cotornionness ${\rm Sel}(K,A_f\otimes\Psi_T^{\cc-1})$ and the integral divisibility 
\begin{equation}\label{eq:ES-argument}
{\rm char}_\Lambda\bigl({\rm Sel}(K,A_f\otimes\Psi_T^{\cc-1})^\vee\bigr)\supset\bigl(\Theta_{f/K}(T)^2\bigr),
\end{equation}
and the converse divisibility that follows from the two-variable main conjecture for $f/K$ proved in the work of Skinner--Urban and Wan \cite{SU,wan-HMF} (after inverting $p$ in the latter case) restricted to the anticyclotomic line. 
\end{proof}

\begin{rem}
By using the Euler system divisibility from \cite[Thm.~5.6.1]{C-Do} (using diagonal cycles rather than congruence arguments and Heegner points), the ramification hypothesis on $\bar\rho_f$ in Theorem~\ref{thm:BD-SU} can be replaced by the ``big image'' hypothesis in [\emph{op.\,cit.}, \S{3.3.2}].
\end{rem}

\section{Main result}

In this section we state and prove the main result of this note toward the nonvanishing conjectures of \cite{DR2.5} in the setting of rank two elliptic curves. 

\subsection{Generalised Kato classes}\label{subsec:GKC}


Let $E/\Q$ be an elliptic curve of conductor $N_f$, and let $p>3$ a prime of good ordinary reduction for $E$. Let $f\in S_2(\Gamma_0(pN_f))$ be the ordinary $p$-stabilisation of the newform associated to $E$, and let $(\bfg,\bfg^*)=(\boldsymbol{\theta}_\psi(S_1),\boldsymbol{\theta}_{\psi^{-1}}(S_2))$ be a dual pair of primitive CM Hida families as in \eqref{eq:gg*}. 
When $\chi=\psi/\psi^{\cc}$ satisfies \eqref{eq:chi-dist}, we have
\begin{equation}\label{eq:shapiro-sp}
\rH^1(\Q,\Vsdagg)\simeq\rH^1(K,T_f\otimes\Psi_{T}^{1-\cc})\oplus\rH^1(K,T_f\otimes\chi^{-1}\Psi_{W_2}^{1-\cc})_{/W_2}
\end{equation}
by specialising \eqref{eq:shapiro}. The following is a special case of the construction of generalised Kato classes by Darmon--Rotger \cite{DR2.5} in the \emph{adjoint case}.

\begin{defn}\label{def:GKC}
Let $\kappa_p(E)$ be the image of $\kappa(f,\ggstar)$ under the composition
\[
\rH^1(\Q,\Vsdagg)\rightarrow\rH^1(K,T_f\otimes\Psi_{T}^{1-\cc})\rightarrow\rH^1(K,T_f)\simeq\rH^1(K,T_pE),
\]
where the first arrow is the projection onto the first direct summand in \eqref{eq:shapiro-sp}, and the second is induced by the multiplication by $T$ on $T_f\otimes\Psi_T^{1-\cc}$. 
\end{defn}

By \eqref{eq:dec-S-gg}, the inclusion $\kappa(f,\ggstar)\in{\rm Sel}^{\rm bal}(\Q,\Vsdagg)$ from \cite[Cor.~8.2]{BSV} implies that 
\begin{equation}\label{eq:inBDP}
\kappa_p(E)\in{\rm Sel}_{\emptyset,0}(K,T_pE).
\end{equation}
In the following we shall view $\kappa_p(E)$ as a class with coefficients in $V_pE$.



\begin{prop}\label{prop:GKC-Sel}
The following implication holds 
\[
L(f/K,1)=0\quad\Longrightarrow\quad\kappa_p(E)\in{\rm Sel}(K,V_pE).
\]
\end{prop}

\begin{proof}
By \eqref{eq:inBDP}, the inclusion $\kappa_p(E)\in{\rm Sel}(K,V_pE)$ holds if and only if ${\rm res}_{\pp}(\kappa_p(E))$ vanishes under the natural map $\rH^1(K_\pp,V_f)\rightarrow\rH^1(K_\pp,V_f^-)$. Since from the combination of Theorem~\ref{thm:ERL},
Proposition~\ref{prop:factor-L}, and Theorem~\ref{thm:BD-theta} yield the equivalence 
\[
{\rm exp}_{\pp}({\rm res}_\pp(\kappa_p(E)))=0\quad\Longleftrightarrow\quad 
L(f/K,1)\cdot L(f/K\otimes\chi,1)=0,
\]
and the Bloch--Kato dual in \eqref{eq:BK} is an isomorphism, the result follows.
\end{proof}

Denote by $E^K$ the twist of $E$ by the quadratic character corresponding to $K/\Q$.

\begin{cor}\label{cor:GKC-Sel-Q}
Suppose $L(E,1)=0$ and $L(E^K,1)\neq 0$. Then $\kappa_p(E)\in{\rm Sel}(\Q,V_pE)$.
\end{cor}

\begin{proof}
Since Kato's results \cite{Kato295} show that ${\rm Sel}(\Q,V_pE^K)=0$ when $L(E^K,1)\neq 0$, this follows from Proposition~\ref{prop:GKC-Sel} and the isomorphism ${\rm Sel}(K,V_pE)\simeq{\rm Sel}(\Q,V_pE)\oplus{\rm Sel}(\Q,V_pE^K)$.
\end{proof}

\begin{rem}
The proof of Proposition~\ref{prop:GKC-Sel} show that when $L(f/K,1)=0$, the class $\kappa_p(E)$ lies in fact in the kernel of the restriction map on ${\rm Sel}(K,V_pE)$ at the primes above $p$. In particular, in the setting of Corollary~\ref{cor:GKC-Sel-Q}, $\kappa_p(E)$ lands in the \emph{strict} Selmer group 
\[
{\rm Sel}_0(\Q,V_pE)={\rm ker}\left\{{\rm Sel}(\Q,V_pE)\xrightarrow{{\rm loc}_p} E(\Q_p)\hat\otimes\Q_p\right\}.
\]
\end{rem}

\subsection{Nonvanishing of $\kappa_p(E)$ in rank two}

Denote by
\[
\rho_f:G_\Q\rightarrow{\rm Aut}_{\mathbf{F}_p}(E[p])\simeq{\rm GL}_2(\mathbf{F}_p)
\]
the mod $p$ Galois representation associated to $E$, and assume that:
\begin{itemize}
\item[(h1)] $\bar{\rho}_f$ is absolutely irreducible,
\item[(h2)] there exists a prime $q\Vert N_f$, and if $q\equiv\pm{1}\pmod{p}$ then $\bar{\rho}_f$ is ramified at $q$,
\item[(h3)] either $p>5$ of the image of $\bar{\rho}_f$ contains a conjugate of ${\rm GL}_2(\mathbb{F}_p)$.
\end{itemize}
We shall consider generalised Kato classes as in \S\ref{subsec:GKC} attached to the following.

\begin{choice}\label{choice:setting}
Let $q$ be a prime as in (h2) above, and let $\ell\nmid 6pN_f$ be a prime. Choose an imaginary quadratic field $K$ and a ring class character $\chi$ of $K$ of conductor dividing $\ell^\infty\cO_K$ such that:
\begin{enumerate}
\item $q$ is inert in $K$,
\item every prime factor of $N_f/q$ splits in $K$,
\item $\ell$ splits in $K$,
\item $\ell$ is ordinary for $E$,
\item $L(E^K,1)\neq 0$ and $L(E/K,\chi,1)\neq 0$,
\item $\chi\vert_{G_{K_\pp}}\neq 1$.
\end{enumerate}
\end{choice}

\begin{rem}
The existence of (infinitely many) $K$ satisfying (1)--(3) and such that $L(E^K,1)\neq 0$ follows from \cite{BFH90}; for any such $K$, the results of \cite{vatsal-special} (see also \cite[Thm.~D]{ChHs1}) ensure that $L(E/K,\chi,1)\neq 0$ for all but finitely many $\chi$ of $\ell$-power conductor. 
\end{rem}


Fix $(K,\chi)$ as in Choice~\ref{choice:setting}. Write $\chi=\psi/\psi^{\cc}$ with $\psi$ a ray class character modulo $\ell^m\cO_K$ for some $m>0$, let $(\bfg,\bfg^*)=(\boldsymbol{\theta}_{\psi}(S_1),\boldsymbol{\theta}_{\psi^{-1}}(S_2))$ be the associated primitive CM Hida families as in (\ref{eq:gg*}), and let $\kappa_p(E)\in\rH^1(K,V_pE)$ the corresponding generalised Kato class. 

By Corollary~\ref{cor:GKC-Sel-Q}, when $L(E,1)=0$ we have $\kappa_p(E)\in{\rm Sel}(\Q,V_pE)$. Moreover, by our choice of $K$ and $\chi$, when further $E/\Q$ has sign $+1$ (so in particular ${\rm ord}_{s=1}L(E,s)\geq 2$), the nonvanishing conjecture of Darmon--Rotger \cite[Conj.~3.2]{DR2.5} (as specialised in [\emph{op.\,cit.}, \S{4.5.3}] to the setting of rank two elliptic curves) predicts the equivalence
\[
\kappa_p(E)\neq 0\quad\overset{?}\Longleftrightarrow\quad{\rm dim}_{\Q_p}{\rm Sel}(\Q,V_pE)=2.
\]
The following is the main result of this note in the direction of this conjecture.

\begin{thm}\label{thm:main}
Suppose that $L(E,s)$ vanishes to positive even order at $s=1$. 
 Then
\[
\kappa_p(E)\neq 0\quad\Longrightarrow\quad{\rm dim}_{\bQ_p}{\rm Sel}(\bQ,V_pE)=2.
\]
Conversely, if ${\rm dim}_{\bQ_p}{\rm Sel}(\Q,V_pE)=2$ then $\kappa_p(E)\neq 0$ if and only if the restriction map
\[
{\rm loc}_p:{\rm Sel}(\bQ,V_pE)\rightarrow E(\bQ_p)\hat\otimes\bQ_p
\]
is nonzero. Moreover, in that case $\kappa_p(E)$ spans the strict Selmer group ${\rm Sel}_0(\Q,V_pE)={\rm ker}({\rm loc}_p)$.
%
\end{thm}

\begin{rem}
Theorem~\ref{thm:main} recovers \cite[Thm.~A]{cas-hsieh-ord} under slightly weaker hypotheses on $\bar{\rho}_f$, and moreover, shows that the condition ${\rm Sel}(\Q,V_pE)\neq{\rm ker}({\rm loc}_p)$ is \emph{necessary} for the nonvanishing of $\kappa_p(E)$ when ${\rm Sel}(\Q,V_pE)$ is $2$-dimensional.
\end{rem}

\begin{rem}
When ${\rm Sel}(\Q,V_pE)$ is $2$-dimensional and different from ${\rm ker}({\rm loc}_p)$, it follows from Theorem~\ref{thm:main} that
\[
\kappa_p(E)=\lambda\cdot\bigl({\rm log}_p(Q)P-{\rm log}_p(P)Q\bigr),
\]
for some $\lambda\in\Q_p^\times$, where $(P,Q)$ is any basis for ${\rm Sel}(\Q,V_pE)$ and ${\rm log}_p$ 
is the composition of ${\rm loc}_p$ with the formal group logarithm   
on $E(\bQ_p)\hat\otimes\Q_p$. 
This is consistent with the refined conjecture in \cite[Conj.~3.12]{DR2.5}  specialised to the setting of [\emph{op.\,cit.},\S{4.5.3}], which further predicts a rationality statement for $\lambda$ (see also \cite[Conj.~3.4, Rem.~3.5.2]{BSV-garrett}).
\end{rem}

The rest of this section is devoted to the proof of Theorem~\ref{thm:main}.

\subsection{Some Galois cohomology} 

Put $K_p=K\otimes_{\Q}\Q_p\simeq K_\pp\oplus K_{\ppbar}$. Let
\begin{equation}\label{eq:locp-H}
{\rm loc}_p={\rm loc}_\pp\oplus{\rm loc}_{\ppbar}:\rH^1(G_{K,\Sigma},V_pE)\rightarrow\rH^1(K_p,V_pE)
\end{equation}
be the restriction map, and denote by $X_f$ (resp. $X_{\pp}$) the image of ${\rm Sel}(K,V_pE)$ (resp. ${\rm Sep}_{\emptyset,0}(K,V_pE)$) under ${\rm loc}_p$.

\begin{lem}\label{lem:-1}
We have ${\rm dim}_{\Q_p}{\rm Sel}_{\emptyset,0}(K,V_pE)={\rm dim}_{\Q_p}{\rm Sel}(K,V_pE)+\delta$, where
\[
\delta=\begin{cases}
-1&\textrm{if $X_f\neq 0$,}\\[0.2em]
0 &\textrm{if $X_f=X_\pp=0$,}
\\[0.2em]
1&\textrm{if $X_f=0$ and $X_{\pp}\neq 0$.}
\end{cases}
\]
\end{lem} 

\begin{proof}
Since $X_f$ is invariant under the action of complex conjugation, if $X_f$ is nonzero then the restriction maps ${\rm Sel}(K,V_pE)\rightarrow\rH^1(K_\qq,V_pE)$ for $\qq\in\{\pp,\ppbar\}$ are both nonzero, and so in this case the result follows from \cite[Lem.\,2.2]{APAW}. 

Now suppose $X_f=0$, so that ${\rm Sel}(K,V_pE)$ is the same as the strict Selmer group ${\rm Sel}_{\rm str}(K,V_pE):={\rm ker}({\rm loc}_p)$, and  consider the exact sequence
\begin{equation}\label{eq:loc-pp}
0\rightarrow{\rm Sel}_{\rm str}(K,V_pE)\rightarrow{\rm Sel}_{\emptyset,0}(K,V_pE)\xrightarrow{{\rm loc}_\pp}\rH^1(K_\pp,V_pE).
\end{equation}
If $X_\pp=0$, then \eqref{eq:loc-pp} yields the result, so it remains to consider the case $X_\pp\neq 0$. If ${\rm dim}_{\Q_p}X_\pp=2$, then the image of \eqref{eq:locp-H} contains $X_\pp\oplus X_{\ppbar}$, where $X_{\ppbar}$ is the image of $X_\pp$ under complex conjugation (equivalently, the image of ${\rm Sel}_{0,\emptyset}(K,V_pE)$), but this contradicts \cite[Lem.\,2.3.1]{skinner}. Therefore if $X_\pp\neq 0$ then it is one-dimensional, and the result in this case follows again from \eqref{eq:loc-pp}.
\end{proof}

\begin{rem}
Standard conjectures predict that $X_f\neq 0$ unless ${\rm Sel}(K,V_pE)=0$ (indeed, the vanishing of $X_f$ when ${\rm Sel}(K,V_pE)\neq 0$ would imply that $\Sha(E/K)[p^\infty]$ is infinite). 
\end{rem}

Recall that we set $\Lambda=\cO\dBr{T}$, and for any $\Lambda$-module $M$ denote by $M_{/T}=M/TM$ the cokernel of multiplication by $T$. The following is a variant of Mazur's control theorem \cite{mazur-18}. 

\begin{prop}\label{prop:control}
Multiplication by $T$ induces natural maps
\begin{align*}
r^*:{\rm Sel}_{0,\emptyset}(K,A_f)&\rightarrow{\rm Sel}_{0,\emptyset}(K,A_f\otimes\Psi_{T}^{\cc-1})[T],\\
r:{\rm Sel}_{\emptyset,0}(K,T_f\otimes\Psi_{T}^{1-\cc})_{/T}&\rightarrow{\rm Sel}_{\emptyset,0}(K,T_f)
\end{align*}
with finite kernel and cokernel.
\end{prop}

\begin{proof}
Letting ${\rm Sel}_{\rm rel}(K,A_f)$ and ${\rm Sel}_{\rm rel}(L,A_f\otimes\Psi_T^{\cc-1})$ be the Selmer groups defined in the same manner as ${\rm Sel}_{0,\emptyset}(K,A_f)$ and ${\rm Sel}_{\emptyset,0}(L,A_f\otimes\Psi_T^{\cc-1})$, respectively, but with the (propagated in the case of $A_f$) relaxed local condition at the primes above $p$, the map $r^*$ fits into the commutative diagram with exact rows
\begin{equation}\label{eq:diag-res}
\resizebox{\displaywidth}{!}{
\xymatrix{
0\ar[r]&{\rm Sel}_{0,\emptyset}(K,A_f)\ar[r]\ar[d]^{r^*}&{\rm Sel}_{\rm rel}(K,A_f)\ar[r]\ar[d]^{s^*}&\rH^1(K_\pp,A_f)\times\frac{\rH^1(K_{\ppbar},A_f)}{\rH^1(K_{\ppbar},A_f)_{\rm div}}\ar[d]^{t^*}\\
0\ar[r]&{\rm Sel}_{0,\emptyset}(K,A_{f}\otimes\Psi_{T}^{\cc-1})[T]\ar[r]&{\rm Sel}_{\rm rel}(K,A_f\otimes\Psi_{T}^{\cc-1})[T]\ar[r]&\rH^1(K_\pp,A_f\otimes\Psi_{T}^{\cc-1})[T]\times\{0\},\nonumber
}}
\end{equation}
where $\rH^1(K_{\ppbar},A_f)_{\rm div}$ denotes the maximal divisible submodule of $\rH^1(K_{\ppbar},A_f)$. 
The map $t^*$ arises from the cohomology long exact sequence associated to
\[
0\rightarrow A_f\rightarrow A_f\otimes\Psi_T^{\cc-1}\xrightarrow{\times T}A_f\otimes\Psi_T^{\cc-1}\rightarrow 0,
\] 
and therefore is surjective with kernel $\rH^0(K_\pp,A_f\otimes\Psi_T^{\cc-1})_{/T}$. Since \cite[Lem.~2.7]{KO} implies that $\#\rH^0(K_\pp,A_f\otimes\Psi_T^{\cc-1})<\infty$, by the Snake Lemma to prove the stated property for $r^*$ it suffices to show that $s^*$ has finite kernel and cokernel. 

The latter map fits into the commutative diagram with exact rows 
\begin{equation}\label{eq:diag-res}
\resizebox{\displaywidth}{!}{
\xymatrix{
0\ar[r]&{\rm Sel}_{\rm rel}(K,A_f)\ar[r]\ar[d]^{s^*}&\rH^1(G_{K,\Sigma},A_f)\ar[r]\ar[d]^{u^*}&\bigoplus_{w\in\Sigma,w\nmid p}\frac{\rH^1(K_w,A_f)}{\rH^1_f(K_w,A_f)}\ar[d]^{v^*}\\
0\ar[r]&{\rm Sel}_{\rm rel}(K,A_f\otimes\Psi_{T}^{\cc-1})[T]\ar[r]&\rH^1(G_{K,\Sigma},A_f\otimes\Psi_{T}^{\cc-1})[T]\ar[r]&\bigoplus_{w\in\Sigma,w\nmid p}\rH^1(K_w,A_{f}\otimes\Psi_{T}^{\cc-1})[T],\nonumber
}}
\end{equation}
where $\rH^1_f(K_w,A_f)$ is the natural  image of $\ker\{\rH^1(K_w,V_f)\rightarrow\rH^1(K_w^{\rm ur},V_f)\}$ (see Definition~\ref{def:Sel-f}). 
The finiteness of $\rH^0(K_\pp,A_f\otimes\Psi_T^{\cc-1})$  implies that of $\rH^0(G_{K,\Sigma},A_f\otimes\Psi_T^{\cc-1})$, and so the map $u^*$ has finite kernel. Since $u^*$ is clearly surjective and by   
\cite[Lem.~3.3]{greenberg-cetraro} (see also \cite[\S{3.3.6}]{JSW}) the kernel of the map $v^*$ is finite, by the Snake Lemma it follows that $s^*$ has the desired properties. 
This shows the result for $r^*$ and the case of $r$ is shown similarly.
\end{proof}

\subsection{Proof of the main result}\label{subsec:proof}

\begin{proof}[Proof of Theorem~\ref{thm:main}]
The decomposition in Corollary~\ref{cor:factor-S} gives
\begin{equation}\label{eq:dec-Sel-A}
X^{\unb}(\Q,\Asdagg)\simeq
{\rm Sel}(K,A_f\otimes\Psi_{T}^{\cc-1})^\vee\oplus\bigl({\rm Sel}(K,A_f\otimes\chi\Psi_{T}^{\cc-1})[T]\bigr)^\vee.
\end{equation}
The same argument as in the proof of Proposition~\ref{prop:control} shows that the natural restriction map
\[
{\rm Sel}(K,A_f\otimes\chi)\rightarrow{\rm Sel}(K,A_f\otimes\chi\Psi_T^{\cc-1})[T]
\]
has finite kernel and cokernel, and therefore from \cite[Cor.~4]{bdIMC} (a consequence of the divisibility in the anticyclotomic main conjecture proved in \emph{op.\,cit.} with $p=\ell$), as extended in \cite{pollack-weston,ChHs2,KPW}), it follows that
\begin{equation}\label{eq:nekovar}
L(E/K,\chi,1)\neq 0\quad\Longrightarrow\quad\#{\rm Sel}(K,A_f\otimes\chi\Psi_{T}^{\cc-1})[T]<\infty.\nonumber
\end{equation}
Together with Theorem~\ref{thm:BD-SU} and Proposition~\ref{prop:factor-L}, in light of \eqref{eq:dec-Sel-A} this shows that $X^{\unb}(\Q,\Asdagg)$ is $\Lambda$-torsion, with
\[
{\rm char}_{\Lambda}\bigl(X^{\unb}(\Q,\Asdagg)\bigr)=\bigl(\mathscr{L}_p^\unb(f,\ggstar)^2\bigr)
\]
in $\Lambda\otimes\Q_p$. 
By Proposition~\ref{prop:equiv}, it follows that $\kappa(f,\ggstar)$ is not $\Lambda$-torsion, that both ${\rm Sel}^{\rm bal}(\Q,\Vsdagg)$ and  $X^{\rm bal}(\Q,\Asdagg)$ have $\Lambda$-rank one, and that
\begin{equation}\label{eq:char-bal}
{\rm char}_{\Lambda}\bigl(X^{\rm bal}(\Q,\AAdag)_{\rm tors}\bigr)={\rm char}_{\Lambda}\bigl({\rm Sel}^{\rm bal}(\Q,\VVdag)/(\kappa(f,\bfg\bfg^*))\bigr)^2\nonumber
\end{equation}
in $\Lambda\otimes\Q_p$. Denote by $\kappa_1(f,\ggstar)\in{\rm Sel}_{\emptyset,0}(K,T_f\otimes\Psi_{T}^{1-\cc})$ the projection of $\kappa(f,\ggstar)$ onto the first direct summand in \eqref{eq:dec-S-gg}. By Proposition~\ref{prop:factor-S} and Corollary~\ref{cor:factor-S}, 
the above shows 
that both ${\rm Sel}_{\emptyset,0}(K,T_f\otimes\Psi_T^{1-\cc})$ and $X_{0,\emptyset}(K,A_f\otimes\Psi_T^{\cc-1})$ have $\Lambda$-rank one, with
%
\begin{equation}\label{eq:char-ideals}
{\rm char}_{\Lambda}\bigl(X_{0,\emptyset}(K,A_f\otimes\Psi_T^{\cc-1})_{\rm tors}\bigr)={\rm char}_{\Lambda}\bigl({\rm Sel}_{\emptyset,0}(K,T_f\otimes\Psi_T^{1-\cc})/(\kappa_1(f,\ggstar))\bigr)^2\nonumber
\end{equation}
in $\Lambda\otimes\Q_p$. By the control theorem of Proposition~\ref{prop:control}, this implies that
\begin{equation}\label{eq:odd}
{\rm corank}_{\Z_p}{\rm Sel}_{0,\emptyset}(K,A_f)={\rm rank}_{\Z_p}\bigl(X_{0,\emptyset}(K,A_f\otimes\Psi_T^{\cc-1})_{/T}\bigr)=1+2\,{\rm rank}_{\Z_p}(\mathfrak{Z}_{/T}),
\end{equation}
where $\mathfrak{Z}={\rm Sel}_{\emptyset,0}(K,T_f\otimes\Psi_T^{1-\cc})/(\kappa_1(f,\ggstar))$. Thus we conclude that
\[
{\rm corank}_{\Z_p}{\rm Sel}_{0,\emptyset}(K,A_f)=1\quad\Longleftrightarrow\quad(\kappa_1(f,\ggstar)\;{\rm mod}\;T)\neq 0,
\]
where $(\kappa_1(f,\ggstar)\;{\rm mod}\;T)$ is the image of $\kappa_1(f,\ggstar)$ in ${\rm Sel}_{\emptyset,0}(K,T_f\otimes\Psi_T^{1-\cc})_{/T}$. Since the natural map 
\[
{\rm Sel}_{\emptyset,0}(K,T_f\otimes\Psi_T^{1-\cc})_{/T}\rightarrow{\rm Sel}_{\emptyset,0}(K,T_f)
\]
has finite kernel by Proposition~\ref{prop:control}, and   
it sends $(\kappa_1(f,\ggstar)\;{\rm mod}\;T)$ to $\kappa_p(E)$ by construction, we arrive at
\begin{equation}\label{eq:GKCnonzero}
\kappa_p(E)\neq 0
\quad\Longleftrightarrow\quad{\rm rank}_{\Z_p}{\rm Sel}_{\emptyset,0}(K,T_f)=1,
\end{equation}
using the action of complex conjugation to reverse the roles of $\pp$ and $\ppbar$. Now, assuming $\kappa_p(E)\neq 0$, by Theorem~\ref{thm:BD-SU} (resp. the $p$-parity conjecture \cite{nekovarII}) the case $\delta=0$ (resp. $\delta=1$) is excluded in Lemma~\ref{lem:-1}, and so ${\rm dim}_{\bQ_p}{\rm Sel}(K,V_pE)=2$; since the nonvanishing of $L(E^K,1)$ implies that ${\rm Sel}(\Q,V_pE)={\rm Sel}(K,V_pE)$ by Kato's work, the first implication in Theorem~\ref{thm:main} follows. 

Conversely, if ${\rm dim}_{\bQ_p}{\rm Sel}(\bQ,V_pE)=2$ (and so ${\rm dim}_{\bQ_p}{\rm Sel}(K,V_pE)=2$), from \eqref{eq:odd} we see that the case $\delta=0$ is excluded in Lemma~\ref{lem:-1} and the case $\delta=-1$ holds (i.e., ${\rm dim}_{\Q_p}{\rm Sel}_{\emptyset,0}(K,V_pE)=1$) if and only if  ${\rm Sel}(\Q,V_pE)\neq{\rm ker}({\rm loc}_p)$. By \eqref{eq:GKCnonzero}, this concludes the proof.
\end{proof}


\begin{rem}
Even though the approach in this paper recovers a strenghtened form of \cite[Thm.\,A]{cas-hsieh-ord} under slightly weaker hypotheses, a noteworthy advantage of the approach in \emph{op.\,cit.} is that --- by relating the derived $p$-adic heights against $\kappa_p(E)$ to the leading coefficient of $\Theta_{f/K}(T)$ (see \cite[Thm.~5.3]{cas-hsieh-ord}) --- it provides a method to numerically verity the nonvanishing of $\kappa_p(E)$, and to relate its rationality properties predicted by \cite[Conj.~3.12]{DR2.5} to  (the leading coefficient part of) the anticyclotomic $p$-adic Birch--Swinnerton-Dyer conjectures formulated in \cite{BDmumford-tate}. (See also \cite{BSV-garrett} for a more general setting.)
\end{rem}

\bibliographystyle{amsalpha}
\bibliography{Schoen-Kato-refs}

\end{document}